\documentclass[11pt,reqno,a4paper]{amsart}

\usepackage{mathtools}

\usepackage{amssymb}

\usepackage{amsfonts}

\usepackage[colorlinks=true, linkcolor=blue, urlcolor=blue, citecolor=blue, pagebackref=true]{hyperref}

\usepackage[abbrev,bibtex-style]{amsrefs}

\numberwithin{equation}{section} \theoremstyle{plain}
\newtheorem{theorem}{Theorem}
\newtheorem{proposition}[theorem]{Proposition}
\newtheorem{lemma}[theorem]{Lemma}
\newtheorem{corollary}[theorem]{Corollary}

\newtheorem{conjecture}[theorem]{Conjecture}

\newtheorem{problem}{Problem}

\theoremstyle{definition}
\newtheorem{remark}[theorem]{Remark}

\renewcommand{\leq}{\leqslant}
\renewcommand{\geq}{\geqslant}

\newcommand\E{\mathbb{E}}
\renewcommand\L{\mathbb{L}}
\newcommand\Z{\mathbb{Z}}
\newcommand\R{\mathbb{R}}

\newcommand\C{\mathbb{C}}
\newcommand\N{\mathbb{N}}

\newcommand\G{\mathbb{G}}

\newcommand\Q{\mathbb{Q}}

\newcommand\Upp{\operatorname{Upp}}

\newcommand\GL{\operatorname{GL}}

\newcommand\Supp{\operatorname{Supp}}

\renewcommand\P{\mathbb{P}}
\newcommand\Id{{\rm Id}}

\newcommand\diag{\operatorname{diag}}
\newcommand\Aff{\operatorname{Aff}}

\renewcommand\Re{\operatorname{Re}}
\renewcommand\Im{\operatorname{Im}}

\renewcommand{\a}{\alpha}

\renewcommand{\d}{\delta}

\newcommand{\e}{\varepsilon}

\renewcommand{\l}{\lambda}

\newcommand{\s}{\sigma}

\newcommand{\wt}{\widetilde}

\newcommand{\cD}{\mathcal{D}}

\newcommand\eps{\varepsilon}

\begin{document}

\title[Bernoulli convolutions and group growth]{Entropy of Bernoulli convolutions and uniform exponential growth for linear groups}

\author{Emmanuel Breuillard}
\address{Centre for Mathematical Sciences\\
Wilberforce Road\\
Cambridge CB3 0WA\\
UK}
\email{emmanuel.breuillard@maths.cam.ac.uk}

\author{P\'eter P. Varj\'u}
\address{Centre for Mathematical Sciences\\
Wilberforce Road\\
Cambridge CB3 0WA\\
UK}
\email{pv270@dpmms.cam.ac.uk}

\thanks{EB acknowledges support from ERC Grant no. 617129 `GeTeMo';
PV acknowledges support from the Simons Foundation and the Royal Society.}

\begin{abstract} The exponential growth rate of non polynomially growing subgroups of $\GL_d$ is conjectured to admit a uniform lower bound. This is known for non-amenable subgroups, while for amenable subgroups it is known to imply the Lehmer conjecture from number theory. In this note, we show that it is equivalent to the Lehmer conjecture. This is done by establishing a lower bound for the entropy of
the random walk on the semi-group generated by the maps $x\mapsto \lambda\cdot x\pm1$, where $\l$ is an algebraic number.
We give a bound in terms of the Mahler measure of $\lambda$. We also derive a bound on the dimension
of Bernoulli convolutions.
\end{abstract}

\maketitle
\section{Introduction}

A Bernoulli convolution with parameter $\lambda$ is the distribution $\mu_\lambda$ of the infinite random series $\sum_{n \geq 0} \pm \lambda^n$, where the $\pm$ are independent fair coin tosses, and $\lambda$ is a real number between $0$ and $1$.
Such measures appear in a large number of situations in harmonic analysis and dynamical systems. A key question regarding them, arguably the most puzzling, is asking for which values of $\lambda$ is this distribution absolutely continuous with respect to Lebesgue measure.
If $\lambda<1/2$, then $\mu_\l$ is singular, being supported on a Cantor set.
Surprisingly, there is a family of $\lambda$'s greater than $1/2$ with singular $\mu_\lambda$
(the inverses of the Pisot numbers in $(1,2)$, see \cite{erdos39}).
It is a well-known problem going back to Erd\H os to determine the set of values of $\l$
for which $\mu_\lambda$ is absolutely continuous,
see \cites{erdos39, 60y}.

If $\mu_\lambda$ is absolutely continuous, then its dimension $\dim \mu_\lambda$ coincides with the dimension of Lebesgue measure, namely $1$. In \cite{hochman}, M. Hochman made a breakthrough in this direction, by establishing that $\dim \mu_\lambda = 1$ unless $\lambda \in (\frac{1}{2},1)$ is \emph{almost algebraic} in the sense that there is a sequence of degree $n$ polynomials $p_n$ with coefficients in $\{-1,0,1\}$ such that $p_n(\lambda)$ tends to $0$ super-exponentially fast.
It is easily seen that the set of almost algebraic numbers has packing dimension zero, and this was further exploited by Shmerkin \cite{shmerkin} to obtain that $\mu_\lambda$ is absolutely continuous for all $\lambda \in (\frac{1}{2},1)$, except perhaps for a subset of $\lambda$'s of Hausdorff dimension zero.
(We recall that a set of $0$ packing dimension is also of $0$ Hausdorff dimension.)

In the first part of this paper we will study the opposite situation when $\lambda$ is assumed to be an algebraic number.
In this case the study of Bernoulli convolutions is closely related to another famous conjecture: the Lehmer conjecture about algebraic numbers.
This asserts that the Mahler measure of an irreducible polynomial in $\Z[X]$
ought to be bounded away from $1$ uniformly, unless it is equal to $1$.

Recall that if $\pi_\lambda:=a_r \prod_1^r (X-\lambda_i) = a_rX^r + \ldots + a_1X + a_0 \in \Z[X]$
is the minimal polynomial of an algebraic number $\lambda \in \overline{\Q}$, then the Mahler measure $M_\lambda$ of $\pi_\lambda$ is defined by
\begin{equation}\label{mahler}
M_\lambda:=|a_r|\prod_{|\lambda_i| > 1} |\lambda_i|.
\end{equation}

We can now state our first result:

\begin{theorem} \label{lehmer}
If the Lehmer conjecture holds, then there is $\eps>0$ such that for every real algebraic number $\lambda$ with
$1-\eps < \lambda < 1$, the dimension of the Bernoulli convolution $\mu_\lambda$ is $1$.
\end{theorem}

Moreover, our methods give unconditional results, too.
We provide an easily testable sufficient (but unfortunately, not necessary) condition
that implies $\dim \mu_\l=1$, which yields plentiful new examples of Bernoulli convolutions with full dimension.
See Theorem \ref{main} below, and the discussion that follows it.

In a follow-up paper \cite{BV-transcendent}, we prove the following result among others.
\begin{theorem}\label{th:transcendent}
We have
\[
\{\l\in(1/2,1):\dim\mu_\l<1\}\subseteq\overline{\{\l\in\overline{\Q}\cap(1/2,1):\dim\mu_\l<1\}},
\]
where $\overline \Q$ is the set of algebraic numbers
and $\overline{\{\cdot\}}$ denotes the closure of the set with respect to the
natural topology of  real numbers.
\end{theorem}

Using this result, we can drop the condition of algebraicity from Theorem \ref{lehmer}.
If Lehmer's conjecture holds, then there is $\eps>0$ such that $\dim\mu_\l=1$ for all $1-\eps<\l<1$.
In addition, Theorem \ref{th:transcendent} provides additional motivation for studying the case of algebraic parameters.
Indeed, if one was able to prove that the inverses of Pisot numbers are the only algebraic parameters
such that $\dim\mu_\l<1$, then we would have $\dim\mu_\l=1$ for all transcendental parameters,
since the set of Pisot numbers is closed (see \cite{Sal-Pisot-closed}).
Unfortunately, we are not able to prove this, but we believe that our methods introduced in this paper
may yield an approach.

In the second part of this paper, we discuss a connection between Bernoulli convolutions and a classical topic in geometric group theory, namely the growth of finitely generated groups. Given a group $\Gamma$ generated by a finite subset $S$, we denote by $\rho_S$ the rate of exponential growth of its $n$-th powers, namely:

$$\rho_S:= \lim_{n \to +\infty} \frac{1}{n}\log|S^n|,$$
where $S^n$ is the set of products of $n$ elements chosen from $S$.

Examples of Grigorchuk and de la Harpe \cite{grigorchuk-delaharpe} show that $\rho_S$ can take arbitrarily small positive values even for linear groups. However in \cite{breuillardicm}*{Conjecture 1.1}, the first author made the following conjecture:

\begin{conjecture}[Growth Conjecture] Given $d \in \N$, there is $\eps=\eps(d)>0$ such that for every finite subset $S$ in $\GL_d(\C)$, either $\rho_S=0$, or $\rho_S>\eps$.
\end{conjecture}

A positive answer to this conjecture was obtained earlier in \cites{breuillard-tits, breuillard-zimmer} in the special case when $S$ generates a non-virtually solvable subgroup of $\GL_d(\C)$. Indeed this is a simple consequence of the uniform version of the Tits alternative proved therein (see also \cites{emo,breuillard-gelander} for earlier related works). A key ingredient in this work was the proof \cite{breuillard-gap} of an analogue of Lehmer's conjecture in the setting of semisimple algebraic groups.

The case when the subgroup generated by $S$ is virtually solvable, i.e. contains a solvable subgroup of finite index, is surprisingly harder, since it was observed in \cite{breuillard-solvable}  that the above conjecture, already in the case of solvable subgroups of $\GL_2$, implies the Lehmer conjecture.

 We can now show the converse:

\begin{theorem}\label{equi} The Growth conjecture is equivalent to the Lehmer conjecture.
\end{theorem}

The key behind the proofs of Theorems \ref{lehmer} and \ref{equi} is the study of the entropy of Bernoulli convolutions with algebraic parameter via Theorem \ref{main} below, which directly relates  the Mahler measure to the entropy.

The ping-pong method used in most proofs of exponential word growth is not powerful enough in our situation. Indeed no free semi-group may in general be generated by words of small length (see \cite{breuillard-solvable}*{Thm 1.7}). Fortunately here entropy comes to our rescue.

While it is clear how entropy relates to an exponential growth rate (Jensen's inequality, see $(\ref{uppp})$), the relation between entropy and the dimension of a Bernoulli convolution is provided by Hochman's theorem alluded above, which tells us that if $\lambda \in \overline{\Q}\cap(1/2,1)$,
then
\begin{equation}\label{hoch}\dim \mu_\lambda = \min\Big(\frac{h_\lambda}{|\log \lambda|},1\Big),\end{equation}
where $h_\lambda$ is the \emph{entropy} of the random walk on the semi-group
generated by the affine transformations $x\mapsto \l\cdot x\pm 1$.
More precisely, it is defined as follows:
\begin{equation}\label{eq:hlambda-def}
h_\lambda:= \lim_{n \to +\infty} \frac{H(\mu_\lambda^{(n)})}{n},
\end{equation}
where $H(\theta)$ denotes the Shannon entropy of the discrete probability measure $\theta$ and $\mu_\lambda^{(n)}$ is
the law of the random variable
$\mu_\l^{(n)}=\sum_{i=0}^{n-1}\pm\l^i$, where $\pm$ are independent fair coin tosses.
If $X_1,\ldots,X_n$ are independent random similarities, which take the values $x\mapsto \l\cdot x+1$ or
$x\mapsto \l\cdot x-1$ with equal probability, then $\mu_\l^{(n)}$ is also the law of $X_n\cdots X_1(0)$.
This explains our terminology of calling $h_\l$ the entropy of the random walk.
We explain in Paragraph \ref{sc:hochman} how \eqref{hoch} follows from \cite{hochman}.

It is convenient for us to use the following convention for the base of logarithms.
\emph{We write $\log$ for the base $2$ logarithm, so that $\log 2=1$.} We can now state our main theorem.

\begin{theorem}\label{main} There is a positive constant $c>0$, such that the following holds.
Let $\lambda$ be an algebraic number.
Then
$$c\min\{1,\log M_\l\}\leq h_\lambda \leq \min\{1,\log M_\l\}.$$
\end{theorem}

We stress that $\lambda$ here can take any complex value in $\overline{\Q}$, not only real values.
Our (non-rigorous) numerical calculations show that one can take $c=0.44$ in the above result.
This is probably not optimal.

\begin{remark}\label{rm:general}
In fact, the result holds in greater generality.
Let $\xi_0,\xi_1,\ldots$ be a sequence of finitely supported i.i.d. random variables with common law $\nu_0$.
Suppose that $\xi_n\in\Q$ almost surely.
Let $\lambda$ be an algebraic unit.
Write $\mu_{\l,\nu_0}^{(n)}$ for the law of the random variable $\sum_{i=0}^{n-1}\xi_i\l^i$ and
write $h_{\l,\nu_0}$ for the quantity that we obtain when we replace $\mu_\l^{(n)}$ by $\mu_{\l,\nu_0}^{(n)}$
in \eqref{eq:hlambda-def}.
Then there is a constant $c(\nu_0)$ depending only on $\nu_0$ such that
\[
c(\nu_0)\min\{1,\log M_\l\}\le h_{\l,\nu_0}\le \min\{H(\nu_0),\log M_\l\}.
\]
Moreover, for each $\e>0$, there is a measure $\nu_0$ such that $c(\nu_0)>1-\e$.
This means, in particular, that for every algebraic unit $\lambda\in(1/2,1)$ such that $\lambda^{-1}$ is neither a Pisot
nor a Salem number, there is a measure $\nu_0$
supported on the integers, such that $\dim \mu_{\lambda,\nu_0}=1$.
(Without the requirement that $\nu_0$ is supported on the integers, this has been known before,
see \cite{SS-complex-bernoulli}*{Theorem D (ii)} for a related result.)

Our main motivation for considering this more general case is an application in the paper \cite{Var-Bernoulli-algebraic}
for the absolute continuity of certain biased Bernoulli convolutions.
\end{remark}

The upper bound in Theorem \ref{main} is often a strict inequality. Indeed we prove:

\begin{proposition}\label{prp:strict}
Let $\l$ be an algebraic number such that $M_\l<2$ and assume that $\l$ has no Galois conjugates
on the unit circle.
Then $h_\l<\log M_\l$.
\end{proposition}

We note that this proposition was proved by Garsia \cite{Garsia-entropy} in the case where $\lambda$ is the inverse of a Pisot number in $(1,2)$.

The upper bound in Theorem \ref{main} follows from a simple counting argument. In fact if we denote by $\rho_\lambda$ the rate of exponential growth (in base $2$) of the size of the support of $\mu_\lambda^{(n)}$, namely

\begin{equation}\label{roelldef}\rho_\lambda := \lim_{n \to +\infty} \frac{\log |\Supp(\mu_\lambda^{(n)})|}{n},\end{equation}
then an obvious upper bound on
$h_\lambda$ is \begin{equation}\label{uppp}h_\lambda \leq \rho_\lambda \leq \log 2=1,\end{equation} as follows say from Jensen's inequality, see $(\ref{support})$.

It is easy to see that $\rho_\l = 1$ unless $\lambda$ is a root of a polynomial with coefficients in $\{-1,0,1\}$.
For a topological study of the set of roots of such polynomials, see \cite{bandt}.

If $\lambda$ is algebraic, then one has the following easy upper bound (see Lemma  \ref{lem:counting})
\begin{equation}\label{growthbound} \rho_\l \leq \log M_\lambda,\end{equation}
from which it also follows that $M_\lambda \geq 2$ unless  $\lambda$ is  a root of a polynomial with $\{-1,0, 1\}$ coefficients,
a fact known since \cites{pathiaux,mignotte}.

For the strict inequality in Proposition \ref{prp:strict}, one needs a slightly more precise upper bound on the size of the support of $\mu_\l^{(\ell)}$, which is obtained by first showing that the law of the image of $\mu_\l^{(\ell)}$ under the geometric embedding $\Q(\l) \to \R^{n+2m}$ (where $n$ is the number of real embeddings, $2m$ the number of complex embeddings,  and $\l$ has modulus less than $1$) is singular with respect to Lebesgue measure. It is an adaptation of Garsia's proof \cite{Garsia-entropy} of this proposition in the special case when $\l^{-1}$ is Pisot.

Theorem \ref{lehmer} is a direct corollary of Theorem \ref{main} and Hochman's identity $(\ref{hoch})$. We note that Theorem \ref{main} also has unconditional consequences, since Lehmer's conjecture has been verified for several classes of algebraic numbers, see e.g. \cites{amoroso-david,borwein}.
Note that already Hochman's identity readily implies that $\dim \mu_\l =1$, provided $\l\in \overline{\Q} \cap (\frac{1}{2},1)$ and the semi-group generated by $x\mapsto\l\cdot x\pm1$
is free, e.g. when $\l$ is not a unit or has a conjugate outside the annulus $1/2\le |z|\le2$.
Theorem \ref{main} provides another condition implying $\dim \mu_\l =1$, which can be easily tested, namely:

\begin{corollary} If $\l$ is a real algebraic number such that
\[
\min(M_\l,2)^{-0.44} \leq \l\leq1,
\]
then $\dim \mu_\l =1$.
\end{corollary}

Here we used are numerical estimate $c_0\ge0.44$ on the constant in Theorem~\ref{main}.

The derivation of Theorem \ref{equi} from Theorem \ref{main} is based on some group theoretic arguments, which enable one to give a lower bound on the growth rate of an arbitrary virtually solvable subgroup of $\GL_d(\C)$ in terms of the growth rate of the semi-group generated by the two affine transformations of the complex line $x \mapsto \l x+1$ and $x \mapsto \l x -1$, which is precisely $\rho_\l$. See Theorem \ref{solgrowth} in Section \ref{uniform}.

The proof of the lower bound in Theorem \ref{main} is the main contribution of this paper. Our argument is a multi-scale analysis that exploits the self-similarity (see \eqref{ssi} below) of the measure $\mu_\l$, together with an inequality for the entropy of the sum of independent random variables (see Proposition \ref{ruzsa}), which can be seen as an entropy analogue of the Pl\"unnecke-Ruzsa inequality from additive combinatorics and is discussed in \cites{tao, madiman, kontoyannis-madiman}. This allows to lower bound the entropy by a sum of entropy contributions at each scale, each of which is uniformly bounded below.

Finally, we mention some related works from the literature.

In \cite{hare-sidorov} Hare and Sidorov computed explicit lower bounds for the entropy $h_\l$ in the special case, when $\l$ is a Pisot number. In particular they showed, that $h_\l > 0.81\cdot \log\l^{-1}$ for all Pisot numbers, which is a better constant than what our methods yield. Moreover, Sidorov pointed out to us that their result holds in greater generality for all algebraic numbers.
On the other hand, Theorem \ref{main} provides an estimate for $h_\l$ in terms of the Mahler measure with a constant
independent of the number of Galois conjugates inside the unit disk. This uniformity is the main difficulty and the main point in our work. See also \cite{alexander-zagier} for upper and lower bounds on $h_\l$, in the case when $\l$ is the Golden Ratio.

Paul Mercat studied the quantity $\rho_\l$ in his thesis \cite{mercat}.
He showed that $\rho_\l= \log M_\l$, whenever $\l\in(1/2,1)$ is a Salem number.
This is also true for Pisot numbers  \cite{lalley}.
Mercat also showed that when $\l$ has no conjugates on the unit circle, then the semi-group
generated by $x\to \l x\pm1$ is automatic,
in particular $2^{\rho_l}$ is an algebraic number in this case,
and he gave an algorithm to compute its minimal polynomial.

The following result of Peters \cite{Peters} was brought to our attention by Andreas Thom.
For each integer $N$ and algebraic number $\l$, denote by $S_{\l,N}$ the set of matrices
\[
\left(
  \begin{array}{cc}
  \l & n \\
        0 & 1 \\
   \end{array}
\right),
\]
where $n$ ranges through the integers $-N,\ldots, N$.
It is proved in \cite{Peters} that
\[
\lim_{N\to\infty}\rho_{S_{\l,N}}=\log M_\l.
\]
The main interest in our results is that we are able to give a good lower bound even for $\rho_{S_{\l,1}}$.

\bigskip

The paper is organized as follows. In Section \ref{bounds} we outline the proof of the lower bound in Theorem \ref{main}, recall the basic properties of entropy and prove some entropy inequalities we need. Section \ref{prooff} completes the proof of the lower bound in Theorem \ref{main} and also deals with the upper bound estimates.
In Paragraph \ref{sc:hochman} we explain how \eqref{hoch} follows from \cite{hochman} and deduce Theorem \ref{lehmer}.
In Section \ref{uniform} we discuss the applications to the growth of solvable linear groups and prove Theorem \ref{equi}. Finally the last section is devoted to some open problems.

\subsection*{Acknowledgements} It is a pleasure to thank  Francesco Amoroso, Jean Bourgain,  Mike Hochman, Paul Mercat, Nikita Sidorov and Andreas Thom for interesting discussions, and Ariel Rapaport for his useful comments.
We are grateful to the referee for valuable comments and suggestions, which improved the paper.

\section{Entropy bounds for the random walk}\label{bounds}

In this section, we set notations, review some basic properties of entropy, give an outline of the proof of Theorem \ref{main} and discuss a number of preliminaries.

\subsection{Bernoulli convolutions in matrices} Let $A \in M_d(\R)$ be a matrix with real entries and spectral radius strictly less than $1$.
We fix an atomic probability measure $\nu_0$ supported on a finite subset of the rationals.
Let $\{\xi_n\}_{n \geq 0}$ be a sequence of bounded i.i.d. random variables with common law $\nu_0$.
The power series $\sum_{n\geq 0} A^n$ converges absolutely, and hence the random variable
\begin{equation}\label{xa}X_A:= \sum_{n \geq 0} \xi_n A^n\end{equation} is well-defined and almost surely finite.
We denote its law by $\mu_A$.
Its support is some bounded region of $M_d(\R)$, and it satisfies the
self-similarity relation
\begin{equation}\label{ssi}\mu_A = \mu_A^{(\ell)} * A^\ell \mu_A,\end{equation}
where $\mu_A^{(\ell)}$ is the law of the random variable
$$X^{(\ell)}_A:=\sum_{0 \leq n \leq \ell-1} \xi_n A^n,$$
and $A^\ell \mu_A$ is the push-forward of the measure $\mu_A$ by the linear map $x \mapsto A^\ell x$.
If $A=\l\in M_1(\R)$ then this is consistent with the notation $\mu_\l$ and $\mu_\l^{(\ell)}$ used in the Introduction.
(For brevity we omit the subscript $\nu_0$ in our notation.)

For any given vector $x \in \R^d$, we let  $X_{A,x} := X_A \cdot x$ and $X^{(\ell)}_{A,x}:=X^{(\ell)}_{A} \cdot x $
and denote the associated measures on $\R^d$ by $\mu_{A,x}$ and $\mu^{(\ell)}_{A,x}$.
Similarly to $\mu_A$, this measure satisfies the relation
\begin{equation}\label{self-similar}
\mu_{A,x} = \mu_{A,x}^{(\ell)} * A^\ell \mu_{A,x}.
\end{equation}

\subsection{Outline of the proof of Theorem \ref{main}} In this paper $H(X)$ denotes the entropy of the random variable $X$ taking values in $\R^d$.
More precisely, if $X$ takes only countably many values $\{x_i\}_i$, then $H(X)$ will denote the \emph{Shannon entropy}, i.e.
$$H(X)=-\sum_i p_i\log p_i,$$
where $p_i$ is the probability that $X=x_i$.
If on the other hand the distribution of $X$ is absolutely continuous with respect to Lebesgue measure, then $H(X)$ will denote the \emph{differential entropy}, that is
$$H(X) = -\int f(x)\log f(x) dx,$$
where $f(x)$ is the density of $X$, which is well defined provided $f\log f$ is in $\L^1(\R^d)$.
Although we will use the same letter for the Shannon and the differential entropy,
it should not cause confusion: we will only consider the entropy of random variables whose law is either atomic,
and we will then use the Shannon entropy, or absolutely continuous with respect to Lebesgue,
in which case the differential entropy will be used.

We fix an algebraic number $\l \in \overline{\Q}$ and assume that $A\in M_d(\R)$ is such a matrix whose eigenvalues coincide with the
Galois conjugates $\l_i$ of $\l$ of modulus $<1$. Note that the spectral radius of such an $A$ is strictly less than $1$. Of course it is always possible to find such a matrix, because the non-real Galois conjugates, i.e. the roots of the minimal polynomial $\pi_\l$, come in pairs of complex conjugates.

We make the following simple observation: if $\lambda$ satisfies an equation of the form
$\sum_0^{\ell-1} \eps_n \lambda^n = \sum_0^{\ell-1} \eps'_n \lambda^n$
for two sets of rationals $\{\eps_n\}_{0 \leq n \leq \ell-1}$ and $\{\eps'_n\}_{0 \leq n \leq \ell-1}$,
then every Galois conjugate of $\lambda$ will satisfy the same equation,
and hence (looking at a basis of eigenvectors of $A$) we also have for all $x \in \R^d$
$$ \sum_{n=0}^{\ell-1} \eps_n A^n x = \sum_{n=0}^{\ell-1} \eps'_n A^n x.$$
{}From this we can deduce that:
\begin{equation}\label{Abound}
H(X_\lambda^{(\ell)}) = H(X_A^{(\ell)}) \geq H(X_{A,x}^{(\ell)}).
\end{equation}

Our goal (towards Theorem \ref{main}) is to obtain a lower bound on $H(X_{A,x}^{(\ell)})$.
We shall shortly describe a method to do this, but first introduce some more notation.
For a matrix $B \in \GL_d(\R)$ and a bounded random variable $X$ on $\R^d$,
$H(X;B)$ will denote the quantity
$$H(X;B):=H(X+G_B) - H(G_B),$$
where $G_B$ is a centered gaussian random variable with co-variance matrix $B\prescript{t}{}{B}$ that is independent of $X$.
(Here $\prescript{t}{}{B}$ denotes the transpose of $B$.)
If $B_1,B_2 \in M_d(\R)$ are two matrices, then $H(X; B_1 |B_2)$ will denote
$$H(X; B_1 |B_2):=H(X; B_1) - H(X ; B_2).$$

These quantities have the following intuitive meaning.
Denote by $\Delta$ the unit ball in $\R^d$.
Then $H(X;B)$ measures the amount of information needed to describe the law of $X$ up to an error inside $B(\Delta)$.
One may define similar quantities using other mollifiers than the standard gaussian (say,  the indicator function
of $\Delta$ or some other bump like function), which would capture the same intuitive meaning.
These quantities differ from  $H(X;B)$ by an additive constant depending on the dimension.
However, our arguments cannot tolerate any losses, so the choice of the smoothing is important
for technical reasons.

We use the following partial order on $M_d(\R)$.
We write $B_1\preceq B_2$ if $B_2\prescript{t}{}{B_2}-B_1\prescript{t}{}{B_1}$ is a non-negative semi-definite matrix,
or equivalently $\|\prescript{t}{}{B_1}x\|_2\le\|\prescript{t}{}{B_2} x\|_2$ for all $x\in \R^d$.
The quantities defined above enjoy the following properties, which will be crucial for us.
\begin{lemma}\label{entropic}
Let $B_1,B_2 \in \GL_d(\R)$ be such that $B_1 \preceq B_2$.
 Assume that  $X,Y$ are two bounded independent random variables taking values in $\R^d$.  Then
\begin{itemize}
\item[(i)] $H(X;B_1) \geq 0$,
\item[(ii)] $H(X;B_1) + H(Y;B_1) \geq H(X+Y;B_1)$,
\item[(iii)] $H(X;B_1) \geq H(X;B_2),$
\item[(iv)] $H(X+Y; B_1 |B_2 ) \geq H( X ; B_1 |B_2).$
\end{itemize}
\end{lemma}

This lemma will be proved in Section \ref{sec:madiman}. For now we pursue our outline of
the proof of Theorem \ref{main}.
Recall that we started with a matrix $A \in M_d(\R)$ whose eigenvalues coincide with the Galois conjugates
of $\l$ of modulus strictly less than $1$. Assume now that its operator norm (for the canonical Euclidean structure on $\R^d$) is less or equal to $1$. This ensures that $$A^{i+1} \preceq A^i$$ for each $i \in \N$.

The first step in the proof will be to  approximate $H(X_{A,x}^{(\ell)})$
by $H(X_{A,x};A^\ell|\Id)$, where $\Id$ denotes the identity matrix.
Intuitively this is expected to hold, since $X_{A,x}$ differs from $X_{A,x}^{(\ell)}$
only at a scale proportional to $A^\ell(\Delta)$.
In fact, we will show that even  $H(X_{A,x}^{(\ell)})\geq H(X_{A,x};A^\ell|\Id)$
holds.

Then we write
\[
H(X_{A,x};A^\ell|\Id)=\sum_{i=1}^\ell H(X_{A,x};A^i|A^{i-1}).
\]
Using Lemma \ref{entropic} (iv) and the self-similarity property \eqref{self-similar}, we will bound from below each term  on the right hand side by $H(X_{A,x};A|\Id)$. Taking the limit $\ell\to \infty$ we then obtain
\begin{equation}\label{eq:firstbound}
h_{\l,\nu_0} = \lim \frac{H(X_{A}^{(\ell)})}{\ell} \geq H(X_{A,x};A|\Id).
\end{equation}

To estimate the right hand side of \eqref{eq:firstbound}, we use inequality (iv) from Lemma \ref{entropic}
and keep only one convolution factor in $X_{A,x}$.
This yields the following inequality, which is valid for all vectors $x \in \R^d$, and all real matrices $A$ with norm at most $1$ and eigenvalues equal to the conjugates of $\l$ of modulus less than $1$,
\begin{equation}\label{eq:secondbound}
h_{\l,\nu_0}\geq H(\xi_0 \cdot x;A|\Id).
\end{equation}
By a suitable choice of the vector $x$ and the matrix $A$ we reduce this to entropies on $\R$:
\begin{equation}\label{eq:thirdbound}
h_{\l,\nu_0}\geq \Phi(\wt M_\l):=\sup_{t>0}\{H(t\wt M_\l\xi_0+G)-H(t\xi_0+G)\},
\end{equation}
where $G$ is a standard gaussian random variable independent of $\xi_0$, and
\[
\widetilde M_\lambda = \prod_{|\lambda_i|<1} |\lambda_i|^{-1}
\]
with the product is running over the conjugates of $\lambda$ of modulus less than $1$.

We note that $\widetilde M_\lambda$ equals the Mahler measure $M_\lambda$ if $\lambda$ is an algebraic unit, which we
can always assume when $\nu_0=(\delta_{-1}+\delta_1)/2$, including in the context of
Theorems \ref{lehmer}, \ref{th:transcendent}, \ref{equi} and \ref{main}.
Indeed, if the distribution of $\sum_{j=0}^n\xi_j\lambda^n$ is not the normalized counting measure on $2^n$ distinct points, then
\[
a_0+a_1\lambda+\ldots+a_n\lambda^n=b_0+b_1\lambda+\ldots+b_n\lambda^n
\]
for some $(a_0,\ldots,a_n)\neq(b_0,\ldots,b_n)\in\{-1,1\}^n$.
Then $\lambda$ is a root of the polynomial
\[
\frac{a_0-b_0}{2}+\frac{a_1-b_1}{2}x+\ldots+\frac{a_n-b_n}{2}x^n,
\]
which has coefficients $-1$, $0$ and $1$,
and hence $\lambda$ is an algebraic unit.

Then a calculus exercise allows to get the lower bound $c(\nu_0) \min\{1,\log \wt M_\l\}$ for the right hand side above, thus concluding our outline of the proof of Theorem \ref{main}.

In the remainder of this section, we recall some basic facts about Shannon and differential entropies in Paragraph  \ref{sec:basic}, then  prove Lemma \ref{entropic} in Paragraph  \ref{sec:madiman}. The proof of Theorem \ref{main} is given in the next section following the above outline.

\subsection{Basic properties of Entropy}\label{sec:basic}

Recall that we denote by $H(X)$ the Shannon entropy of $X$ if $X$ is a discrete random variable in $\R^d$
and the differential entropy if $X$ is absolutely continuous with respect to the Lebesgue measure on $\R^d$.
We refer the reader to \cite{cover-thomas} for a thorough introduction to information theory and entropy.
The purpose of this paragraph is to recall a few properties.

The Shannon entropy is always non-negative.
The differential entropy on the other hand can take negative values.
For example, if $A \in \GL_d(\R)$, and $X$ is a random variable
with finite differential entropy $H(X)$, then it follows from the change of variables formula that
\begin{equation}\label{change}
H(AX) = H(X) + \log |\det A|,
\end{equation}
which can take negative values when $A$ varies.
On the other hand, if $X$ takes countably many values, the Shannon entropy of $AX$ is the same as that of $X$.
Note that both entropies are invariant under translation by a constant in $\R^d$.

The density of a centered gaussian random variable
$G_{A}$ with co-variance matrix $A\prescript{t}{}{A}$ on $\R^d$ is
\[
g_A(x):=\frac{1}{(2\pi)^{d/2}|\det A|}\cdot e^{-||A^{-1}x||^2/2},
\]
and its  differential entropy is $\frac{d}{2}\log 2e\pi + \log |\det A|$.
It maximizes the differential entropy of a random variable in $\R^d$ with the same co-variance.
For a proof see \cite{cover-thomas}*{Example 12.2.8}.

We define $F(x):=-x\log (x)$ for $x>0$ and recall that
$F$ is concave,
and it is sub-additive, i.e. $F(x+y) \leq F(x)+F(y)$, and it also
satisfies the identity $F(xy)=xF(y)+yF(x)$.

{}From the concavity of $F$ and Jensen's inequality,
we see that for any atomic random variable $X$ taking at most $N$ possible different values,
\begin{equation}\label{support}
H(X) \leq \log N.
\end{equation}

Let now $X$ and $Y$ be two independent random variables in $\R^d$.
If both are atomic, it follows immediately from the sub-additivity of $F(x)$ and the identity  $F(xy)=xF(y)+yF(x)$
that $H(X+Y) \leq H(X) + H(Y)$ for  Shannon entropy.
This is no longer true for differential entropy (since the formula is not invariant under a linear change of variable).
However if $X$ is atomic and bounded, while $Y$ is assumed absolutely continuous, then
\begin{equation}\label{subadd}
H(X+Y) \leq H(X) + H(Y),
\end{equation}
where $H(X)$ is Shannon's entropy and the other two are differential entropies.
To see this,  note that if $f(y)$ is the density of $Y$,
then the density of $X+Y$ is $\E(f(y-X))= \sum_i p_if(y-x_i)$, hence:
\begin{align*}
H(X+Y)&= \int F\Big(\sum_i p_if(y-x_i)\Big) dy\\
&\leq \int \sum_i F(p_i f(y-x_i)) dy\\
&= \int \sum_i F(p_i) f(y-x_i) dy + \int \sum_i p_i F(f(y-x_i)) dy \\
&= \sum_i F(p_i) + \int  F(f(y)) dy = H(X) + H(Y)
\end{align*}
and $(\ref{subadd})$ follows.

In the other direction, we always have the lower bound
\begin{equation}\label{lowerbound}
H(X+Y) \geq \max\{H(X),H(Y)\}
\end{equation}
if all three entropies are of the same type (i.e. either Shannon or differential), as follows easily from the concavity of $F$.

The \emph{relative entropy} (or Kullback-Leibler divergence)
of two absolutely continuous probability distributions $p$ and $q$ on $\R^d$ is defined as
$$D(p||q):= \int p(x) \log \frac{p(x)}{q(x)} dx$$

This quantity is always non-negative (\emph{information inequality})
as follows immediately from Jensen's inequality and the concavity of $x \mapsto \log x$, since
$$-D(p||q)= \int p(x) \log \frac{q(x)}{p(x)} dx \leq \log \Big( \int p(x) \frac{q(x)}{p(x)} dx \Big)  = 0.$$
Moreover $D(p||q)=0$ if and only if $p$ and $q$ coincide almost everywhere.

A direct consequence of this inequality is the fact that the entropy of the joint law
of two random variables $(X,Y)$ is at most the sum of the entropy of each marginal, namely:
$$H(X,Y) \leq H(X) + H(Y)$$
Indeed the difference $H(X) + H(Y) - H(X,Y)$ (also called the \emph{mutual information})
can be expressed as the relative entropy
\[D(p_{(X,Y)}(x,y)||p_X(x)p_Y(y)),\]
where $p_{(X,Y)}(x,y)$ is the joint density of $(X,Y)$
and $p_X(x)$ and $p_Y(y)$ the marginals, i.e. the densities of $X$ and $Y$ respectively.
Equality holds if and only if $X$ and $Y$ are independent.

A similar inequality is as follows.
Suppose $X,Y,Z$ are three $\R^d$-valued absolutely continuous random variables such
that the triple $(X,Y,Z)$ as well as the individual marginals $X$, $Y$ and $Z$ have finite differential entropy.
Then
\begin{equation}\label{mast}
H(X,Y,Z) + H(Z) \leq H(X,Z) + H(Y,Z)
\end{equation}
Indeed, the difference $H(X,Z) + H(Y,Z) - H(X,Y,Z) - H(Z)$ is exactly the relative entropy $D(p||q)$
of the two $(\R^d)^3$-valued probability distributions $p(x,y,z):=p_{(X,Y,Z)}(x,y,z)$
and $q(x,y,z):=p_{(X,Z)}(x,z) p_{(Y,Z)}(y,z)/p_Z(z)$.

\subsection{Inequalities for entropies of sums of random variables}\label{sec:madiman}

The purpose of this section is to prove Lemma \ref{entropic}.
We first recall the following result from \cite{madiman}*{Theorem I.}.

\begin{theorem}[submodularity inequality]
\label{ruzsa}
Assume that $X,Y,Z$ are three independent $\R^d$-valued random variables
such that the distributions of $Y$, $X+Y$, $Y+Z$ and $X+Y+Z$ are absolutely continuous with respect
to Lebesgue measure and have finite differential entropy.
Then
\begin{equation}\label{ruzsaineq}
H(X+Y+Z) + H(Y) \leq H(X+Y) + H(Y+Z).
\end{equation}
\end{theorem}

This result goes back in some form at least to a
paper by Kaimanovich and Vershik \cite{kaimanovich-vershik}*{Proposition 1.3},
which related the positivity of the entropy of a random walk on a group to the existence of bounded harmonic functions.
The version in that paper assumes that the laws of $X$, $Y$ and $Z$ are identical.
The inequality was rediscovered by Madiman \cite{madiman}*{Theorem I.} in the
greater generality stated above.
Then it was recast in the context of entropy analogues of sumset
estimates from additive combinatorics by  Tao \cite{tao}
and Kontoyannis and Madiman \cite{kontoyannis-madiman}.
And indeed Theorem \ref{ruzsa} can be seen as an entropy analogue
of the \emph{Pl\"unnecke--Ruzsa inequality} in additive combinatorics (see \cite{tao-vu}*{Corollary 6.29}).

We provide the short proof for the reader's convenience.
We first give it under the additional assumption that $X$ and $Z$ are absolutely continuous and have finite differential entropy.

\begin{proof}[Proof of Theorem \ref{ruzsa} assuming that $X,Z$ have finite differential entropy]
We apply $(\ref{mast})$ to the random variables $X'=X, Y'=X+Y, Z'=X+Y+Z$ to get
$$H(X,X+Y,X+Y+Z) + H(X+Y+Z) \leq H(X,X+Y+Z)  + H(X+Y, X+Y+Z)$$
However $H(X,X+Y,X+Y+Z)= H(X,Y,Z)$ because the linear transformation
used here has determinant $1$ (see $(\ref{change})$),
while $H(X,X+Y+Z)= H(X,Y+Z) = H(X) + H(Y+Z)$ by independence of $X,Y,Z$.
Similarly  $H(X+Y, X+Y+Z)= H(X+Y) + H(Z)$, and $H(X,Y,Z) = H(X)+H(Y) + H(Z)$.
The result follows.
\end{proof}

For the general case we need to approximate $X$ and $Z$ by absolutely continuous random variables.
We will replace them by $X+\e G_\Id$ and $Z+\e G'_\Id$, where $G_\Id$ and $G_\Id'$ are two independent (of everything)
gaussian random variables with covariance matrices $\Id$.
We will need the following two simple Lemmata.

\begin{lemma}\label{lem:mollify}
Let $X$ be an absolutely continuous bounded random variable in $\R^d$ with finite differential entropy.
Let $G_\Id$ be a standard gaussian random variable in $\R^d$ that is independent of $X$.
Then $H(X)=\lim_{\e\to 0} H(X+\e G_\Id)$
\end{lemma}
\begin{proof}
We have $H(X)\le H(X+\e G_\Id)$ by \eqref{lowerbound} for all $\e$.

Denote by $f$ the density of $X$ and by $g_\e$ the density of $\e G_\Id$.
We observe that $f*g_\e$ converges to $f$ almost everywhere and we set out to construct
a majorant function.
We fix a sufficiently large real number $R$ such that $f(x)=0$ for all $\|x\|>R-10$.
We observe that for all $x$ with $\|x\|>R$ and $0<\e<1$ we have $f*g_\e(x)<e^{-(\|x\|-R)^2}$.
We write $m$ for the maximum of the function $F(x)=-x\log(x)$ and define
$M(x)=m$ for $\|x\|\le R+10$ and $M(x)=(\|x\|-R)^2e^{-(\|x\|-R)^2}$ for $\|x\|> R+10$.

We observe that $F(f*g_\e(x))\le M(x)$ for all $x$, hence
by Fatou's Lemma we have
\[
H(X)\ge \limsup_{\e\to 0}\int F(f*g_\e(x)) dx = \limsup_{\e\to 0} H(X+\e G_\Id).
\]
This completes the proof.
\end{proof}

If $X$ is a random variable and $A$ is an event with positive probability,
then we denote by $X|_A$  a random variable
that satisfies $\P(X|_A\in B)=\P({X\in B}\cap A)/\P(A)$ for any measurable subset $B$ of the domain of $X$.

\begin{lemma}\label{lem:restrict}
Let $X$ be an absolutely continuous  random variable in $\R^d$ with finite differential entropy.
Let $\{A_n\}$ be an increasing sequence of events such that $$\lim_{n\to \infty}\P(A_n)=1.$$
Then $H(X)=\lim_{n\to\infty} H(X|_{A_n})$.
\end{lemma}

\begin{proof}
Let $M$ be a monotone increasing function on $\R_{>0}$ such that
$M(x)=|x\log(x)|$ if $x\le 1/10$ or $x\ge 10$ and $M(x)\ge |x\log(x)|$ for all $x$.
We write $f$ for the density of $X$ and $\P(A_n)^{-1}f_n$ for the density of $X|_{A_n}$.
We note that $f_n$ converges to $f$ almost everywhere and that $|f_n\log(f_n)|\le M(f)$.
Since  $M(x)\le Cx+|x\log x|$ for some constant $C$ and both $f$ and $|f\log(f)|$ are
in $L^1$, we see that $M(f)\in L^1$.
Thus by the dominated convergence theorem, we have
\[
H(X)=\lim_{n\to\infty} -\int f_n(x)\log f_n(x) dx
=\lim_{n\to\infty} H(X|_{A_n}).
\]
\end{proof}

\begin{proof}[Proof of Theorem \ref{ruzsa} in the general case]
First we assume that $X,Y$ and $Z$ are all bounded.
Let $G_\Id$ and $G'_\Id$ be two independent (from each other and $X,Y,Z$)
standard gaussian random variables of dimension matching that of $X$.
We apply the already proved case of the theorem for the random variables
$X+\e G_\Id$, $Y$ and $Z+\e G'_\Id$ and obtain
$$H(X+Y+Z+\eps G_\Id + \eps G'_\Id) + H(Y) \leq H(X+Y+ \eps G_\Id) + H(Y+Z+ \eps G_\Id').$$
In light of Lemma \ref{lem:mollify}, letting $\e\to 0$ we can conclude the theorem for $X,Y,Z$.

If any of $X,Y$ or $Z$ is unbounded, then we define $A_n$ to be the event $\|X\|+\|Y\|+\|Z\|<n$.
Using the theorem for the bounded variables $X|_{A_n}$, $Y|_{A_n}$ and $Z|_{A_n}$ we get
\[
H([X+Y+Z]|_{A_n}) + H(Y|_{A_n}) \leq H([X+Y]|_{A_n}) + H([Y+Z]|_{A_n}).
\]
We take the limit $n\to \infty$ and conclude the theorem from Lemma \ref{lem:restrict}.
\end{proof}

Recall the definition of $H(X;B)$ and $H(X;B_1|B_2)$ from the two paragraphs before Lemma \ref{entropic}.
We finish this section by proving the properties of these entropies claimed in this lemma.

\begin{proof}[Proof of Lemma \ref{entropic}]
Item (i) follows easily from the concavity of $F(x)=-x\log(x)$, indeed
$$H(X+G_{B_1})=\int F(\E(g_{B_1}(x-X))) dx \geq \int \E(F(g_{B_1}(x-X)))dx = H(G_{B_1}).$$

Item (ii) is a consequence of $(\ref{ruzsaineq})$ applied to the three independent variables $X'=X$, $Y'=G_{B_1}$ and $Z'=Y$.

Since $B_1 \preceq B_2$, there exists $M \in M_d(\R)$ such that
$B_2\prescript{t}{}{B_2}=B_1\prescript{t}{}{B_1} + M\prescript{t}{}{M}$.
In particular if $G_{B_1}$ and $G_M$ are two independent centered gaussian distributions on $\R^d$
with co-variance matrix $B_1\prescript{t}{}{B_1}$ and $M\prescript{t}{}{M}$ respectively,
then $G_{B_1}+G_M$ is a centered gaussian variable with co-variance $B_2\prescript{t}{}{B_2}$.

Now  item  (iii) follows from  $(\ref{ruzsaineq})$ by setting
$Y=G_{B_1}$ and $Z=G_M$,
while item (iv) also follows from $(\ref{ruzsaineq})$ applied to the independent
variables $X',Y',Z'$ defined by $X':=Y$, $Y':=X+G_{B_1}$, $Z':=G_M$.
\end{proof}

\section{Proof of Theorem \ref{main} and Proposition \ref{prp:strict}}\label{prooff}

In this section we establish Theorem \ref{main} and Proposition \ref{prp:strict}. In Paragraph \ref{sec:bounds} we give the details of the above outline and give a lower bound for $h_\l$ in terms of $\wt M_\l$ only, via the function $\Phi$ defined above in \eqref{eq:thirdbound}. In Paragraph \ref{sec:proofmainthm} we study the function $\Phi$ and deduce the desired lower bound in Theorem \ref{main}. Finally, we prove the upper bounds of Theorem \ref{main} and Proposition \ref{prp:strict} in Paragraph \ref{sec:upper}.

\subsection{Lower bounds on the entropy}\label{sec:bounds}
We keep the notations introduced in the previous section. In particular $\xi_0$ is a random variable with law $\nu_0$, and $G$ is an independent standard gaussian real random variable.
Recall that
\[
\widetilde M_\lambda = \prod_{|\lambda_i|<1} |\lambda_i|^{-1}
\]
with the product is running over the conjugates of $\lambda$ of modulus less than $1$.
We note that $\widetilde M_\lambda$ equals the Mahler measure $M_\lambda$ if $\lambda$ is an algebraic unit, which we
can always assume when $\nu_0=(\delta_{-1}+\delta_1)/2$.
We recall that for $a>0$
$$\Phi_{\nu_0}(a)= \sup_{t>0} \{H(\xi_0 t a + G)  -  H(\xi_0 t + G)\}.$$
The following proposition is the main goal of this paragraph. It establishes the lower bounds \eqref{eq:firstbound}--\eqref{eq:thirdbound} from our outline.

\begin{proposition}\label{prp:bounds}
Let $\l$ be an algebraic number and $A\in M_d(\R)$ a matrix such that
$\|A\|\le 1$ and the eigenvalues of $A$ coincide with the Galois
conjugates of $\l$ of modulus $<1$.
Then for every $x \in \R^d$
\begin{align*}
h_{\l,\nu_0}\geq &H(X_{A,x};A|\Id)
\geq H(\xi_0 \cdot x;A|\Id),\\
h_{\l,\nu_0}\geq &\Phi_{\nu_0}(\wt M_\l).
\end{align*}
\end{proposition}

We recall that  $X_{A,x}:=X_A \cdot x$ for $x \in \R^d$ and $X_A$ was defined in \eqref{xa}.

\begin{proof}
Note that the spectral radius of $A$ is less than $1$ and thus the random variable $X_A$ is well defined and bounded.
Recall that we denote by $\mu_{A,x}$ and $\mu_{A,x}^{(\ell)}$ the laws of the random variables $X_{A,x}$
and $X_{A,x}^{(\ell)}$ respectively.
In what follows it will be convenient for us to write $H(\mu;B)$ for $H(X;B)$,
where $X$ is a random variable with law $\mu$.
In a similar fashion we also use the notation $H(\mu;B_1|B_2)$.

First we observe that
$$H(\mu_{A,x}; A^{\ell} | \Id) \leq H(\mu_{A,x}^{(\ell)}).$$
Indeed
\begin{align*}
H(\mu_{A,x}; A^{\ell} | \Id) &= H(\mu_{A,x}; A^{\ell}) - H(\mu_{A,x}; \Id)\\
&= H(\mu_{A,x}^{(\ell)} * A^{\ell} \mu_{A,x};  A^{\ell}) - H(\mu_{A,x}; \Id) \\
 &\leq  H(\mu_{A,x}^{(\ell)} ;  A^{\ell}) + H( A^{\ell} \mu_{A,x};  A^{\ell}) - H(\mu_{A,x}; \Id)\\
 &\leq  H(\mu_{A,x}^{(\ell)} ;  A^{\ell})\\
 &\leq H(\mu_{A,x}^{(\ell)}),
\end{align*}
where we used the self-similarity relation $(\ref{self-similar})$ on the second line,
the sub-additivity property (ii) in Lemma \ref{entropic} on the third line.
In the fourth line, we used the fact that the gaussian $G_A$
has the same law as $AG_{\Id}$ together with the change of variable formula for the entropy $(\ref{change})$.
Finally the last line  follows from the sub-additivity property $(\ref{subadd})$.

On the other hand, by definition
$$H(\mu_{A,x}; A^{\ell} | \Id) = \sum_{i=1}^{\ell} H(\mu_{A,x}; A^{i} | A^{i-1}).$$
We claim that each term in this sum is bounded below by $H(\mu_{A,x}; A | \Id)$.
This claim follows easily from the self-similarity relation $(\ref{self-similar})$ and estimate (iv) of Lemma \ref{entropic}.
Indeed, our assumption that $||A||\leq 1$ implies that $||A^i y || \leq ||A^{i-1} y||$ for every $y \in \R^d$ and $i \geq 1$,
and thus ensures that $A^i \preceq A^{i-1}$.
This makes the use of Lemma \ref{entropic} (iv) legitimate and shows that
\begin{align*}
H(\mu_{A,x}; A^{i} | A^{i-1}) &= H(\mu^{(i-1)}_{A,x} * A^{i-1} \mu_{A,x}; A^{i} | A^{i-1})\\
&\geq H(A^{i-1} \mu_{A,x}; A^{i} | A^{i-1}) = H( \mu_{A,x}; A | \Id).
\end{align*}

Now to conclude the proof of the first inequality, observe that
\begin{align*}
h_{\l,\nu_0} &= \lim_{\ell \to \infty} \frac{1}{\ell} H(\mu_{\lambda}^{(\ell)})
\geq \lim_{\ell \to \infty} \frac{1}{\ell} H(\mu_{A,x}^{(\ell)}) \\
&\geq  \lim_{\ell \to \infty} \frac{1}{\ell} H(\mu_{A,x}; A^{\ell} | \Id)
\geq H(\mu_{A,x}; A | \Id).
\end{align*}

To deduce the second bound, we use again Lemma \ref{entropic} (iv):
\[
H(\mu_{A,x}; A | \Id)
\geq  H(\mu^{(1)}_{A,x}; A | \Id)= H(\xi_0 \cdot x; A | \Id).
\]

We turn to the proof of the third bound.
To this end, we optimize the parameters $x \in \R^d$ and $A \in M_d(\R)$.
Note that $A$ is allowed to vary among all matrices in $M_d(\R)$ with prescribed spectrum: $\{\lambda_1,\ldots,\lambda_d\}$,
the Galois conjugates of $\lambda$ that lie inside the open unit disc, and such that $||A||\leq 1$.

Exploiting the rotational symmetry of the normalized gaussian law,
we observe that for any two orthogonal matrices $u,v \in O_d(\R)$,
$$H(\xi_0 \cdot x; A|\Id) = H(\xi_0 \cdot ux; uAv|\Id).$$

By the Cartan decomposition, every matrix $A \in M_d(\R)$ can be written as $A=uDv$
for a diagonal matrix $D=\diag(\alpha_1,\ldots,\alpha_d)$
with $\alpha_1 \geq \ldots \geq\alpha_d \geq 0$ and $u,v \in O_d(\R)$.
The $\alpha_i$'s are called the singular values of $A$.
A well-known theorem of Horn (\cite{horn} and \cite{horn-johnson}*{Pb. 2, p 222}) describes the set of values that
can arise as singular values of a matrix with prescribed spectrum.
It follows from this result that one can find a real matrix $A \in M_d(\R)$
whose eigenvalues are as above $\lambda_1,\ldots,\lambda_d$
and whose singular values are
$\alpha_1 = \ldots = \alpha_{d-1}=1$ and
$\alpha_d=|\lambda_1 \cdots \lambda_d| = 1/\wt M_\lambda$.
Note that $\alpha_1=||A||=1$, so this matrix $A$ satisfies our requirements.

In our case we can also find $A$ by the following simple alternative argument: given $v \in \R^d$, consider the quadratic form $q_v(x):=\sum_{n \geq 0} \langle A_0^n x, v\rangle^2$, where $\langle \cdot \rangle$ is the canonical Euclidean scalar product on $\R^d$, and $A_0$ is any diagonalizable matrix with the prescribed eigenvalues. This form is well-defined, because the spectral radius of $A_0$ is less than $1$. Pick $v\in \R^d$ so that the vectors $^tA_0^n v$ span $\R^d$ as $n$ varies among the integers (this is always possible since $A_0$ is invertible and has distinct eigenvalues). Then $q_v$ is positive definite and there exists $B \in \GL_d(\R)$ such that $q_v(x)=\|Bx\|^2$. Now by construction:
$$q_v(x)=\langle x,v \rangle^2 + q_v(A_0x)=\langle x,v \rangle^2 + \|BA_0x\|^2,$$
from which it follows that $\|BA_0B^{-1}x\|^2=\|x\|^2-\langle x,\prescript{t}{}{B}^{-1}v \rangle^2$, which means that at least $d-1$ singular values of $BA_0B^{-1}$ are equal to $1$. The last one is unambiguously determined by the determinant, hence equal to $1/\wt M_\lambda$.

We write $D=\diag(1,\ldots,1,\wt M_\l^{-1})$.
By the above discussion, there are two orthogonal matrices $u,v\in O(d)$ such that $A=uDv$
and hence $D=u^{-1}Av^{-1}$.
We deduce $h_{\l,\nu_0}\geq H(\xi_0\cdot x; D|\Id)$ for all $x \in \R^d$.

We now pick $x$ of the form $x=t e_d$, for some $t \in \R$ and where $e_d=(0,\ldots,0,1)$
is the last element of the canonical basis of $\R^d$. Then
\begin{align*}
H(\xi_0 x; D|\Id) = &H(\xi_0 t e_d + DG_{\Id}) - H(DG_{\Id}) \\
&-  (H(\xi_0 t e_d + G_{\Id}) - H(G_{\Id}))\\
= &H(\xi_0 t \wt M_\lambda e_d + G_{\Id})  -  H(\xi_0 t e_d + G_{\Id})\\
= &H(\xi_0 t \wt M_\lambda + G)  -  H(\xi_0 t + G)
\end{align*}
 where in the last line $G$ is a normalized one-dimensional gaussian random variable.
We used the change of variable formula \eqref{change} to prove the second equality and
the identity $F(xy)=xF(y)+yF(x)$ satisfied by the function $F(x)=-x\log (x)$ to integrate out
the first $d-1$ variables.
This completes the proof of the proposition.
\end{proof}

\subsection{Proof of the lower bound in Theorem \ref{main}}\label{sec:proofmainthm}

In this paragraph, we complete the proof of the lower bound in Theorem \ref{main} and to this aim, we study the function
$$\Phi_{\nu_0}(a)= \sup_{t>0} \{H(\xi_0 t a + G)  -  H(\xi_0 t + G)\}$$
and prove the following estimates:

\begin{lemma}\label{lem:functionPhi}
The function $\Phi_{\nu_0}$ is monotone increasing and we have
\[
\frac{\Phi_{\nu_0}(a)}{\log a}\ge \frac{\Phi_{\nu_0}(a^2)}{\log a^2}
\]
for any $a>1$.
\end{lemma}

This has the following immediate corollary.
\begin{corollary}\label{cor:functionPhi}
We have
\[
\Phi_{\nu_0}(a)\ge c\min\{\log a,1\}
\]
for all $a>1$, where
\begin{equation}\label{eq:c-const}
c=\min_{\sqrt 2\le a\le 2}\{\Phi_{\nu_0}(a)/\log(a)\}.
\end{equation}
\end{corollary}

The constant $c$ can be numerically estimated by calculating $\Phi_{\nu_0}(a)$ for $a$ running
through a sufficiently dense arithmetic progression and using the monotonicity of $\Phi_{\nu_0}$
to estimate it in the intervals between the points of the progression.
Our numerical calculations show $c\ge0.44$ for the case $\nu_0=(\d_{-1}+\d_{1})/2$ that is relevant
for Bernoulli convolutions.
In these calculations, we estimated $\Phi_{\nu_0}(a)$ for $a$ running between $1.4$ and $2$
in increments of $0.01$.
We used the lower bound
\[
\Phi_{\nu_0}(a)\ge H(\xi_0 ta^{1/2}+G)-H(\xi_0ta^{-1/2}+G)
\]
setting $t=1.19$.
This value was selected to optimize the lower bound at $a=2$.
The calculations used MATLAB's built-in routines for numerical evaluation of integrals, which do not provide
error estimates, therefore these calculations are not rigorous.
According to our calculations the function $\Phi_{\nu_0}(a)/\log(a)$ appears to be monotone decreasing,
hence the minimum is probably attained for $a=2$.

We note that for each $\e>0$, it is possible to choose the measure $\nu_0$ in such a way that the constant
\eqref{eq:c-const} is at least $1-\e$.
Indeed, a simple calculation shows that $\Phi_{\nu}(a)=\log a$, for each $a>1$ if $\nu$ is a Gaussian measure.
Hence, one can take for $\nu_0$ a measure supported on the rationals that suitably approximates a Gaussian measure
and find that $c$ is as close to $1$ as desired.

Observe that Corollary \ref{cor:functionPhi} combined with Proposition \ref{prp:bounds}
completes the proof of the lower bound in Theorem \ref{main}.

\begin{proof}[Proof of Lemma \ref{lem:functionPhi}]
Using the change of variable formula \eqref{change} we get
\begin{align*}
H(\xi_0 t a+G)-H(\xi_0 t+G)=&H(\xi_0 t+a^{-1}G)-H(a^{-1}G)\\
&-[H(\xi_0 t+G)-H(G)]\\
=&H(\xi_0t;a^{-1}|1).
\end{align*}
This is an increasing function of $a$ for each fixed $t$ by part (iii) of Lemma \ref{entropic}.
This shows that $\Phi_{\nu_0}(a)$ is an increasing function of $a$.

We turn to the second claim.
Fix $a$ and $\e>0$ and let $t$ be such that
\[
\Phi_{\nu_0}(a^2)\le H(\xi_0 t a^2 + G)  -  H(\xi_0 t + G)+\e.
\]
Then by the definition of $\Phi_{\nu_0}(a)$, we have
\[
\Phi_{\nu_0}(a^2)\le H(\xi_0 t a^2 + G)  - H(\xi_0 t a + G)  +  H(\xi_0 t a+ G) - H(\xi_0 t + G)+\e\le2\Phi_{\nu_0}(a)+\e.
\]
We take $\e\to0$ to conclude $\Phi_{\nu_0}(a^2)\le 2\Phi_{\nu_0}(a)$, which is precisely the second claim.
\end{proof}

\begin{proof}[Proof of Corollary \ref{cor:functionPhi}]
If $a\ge2$, then
\[
\Phi_{\nu_0}(a)\ge\Phi_{\nu_0}(2)\ge c=c\min\{\log a,1\}
\]
by the monotonicity of $\Phi_{\nu_0}$ and the definition of $c$.

If $a<2$, then we take $n\ge 0$ to be an integer such that
$\sqrt2\le a^{2^n}<2$.
Then by the second claim of Lemma \ref{lem:functionPhi} applied repeatedly and the definition of $c$
we have
\[
\frac{\Phi_{\nu_0}(a)}{\log a}\ge\frac{\Phi_{\nu_0}(a^{2^n})}{\log a^{2^{n}}}\ge c.
\]
This shows $\Phi_{\nu_0}(a)\ge c\log a$ proving the claim.
\end{proof}

\subsection{Proof of the upper bounds}\label{sec:upper}

The goal of this paragraph is to prove the upper bound $h_\l\le \rho_\l \le \min\{1,\log M_\l\}$ in Theorem \ref{main}.
We will also show Proposition \ref{prp:strict}, which says that the inequality $h_\l\le \min\{1,\log M_\l\}$  is strict if $M_\l<2$ and $\l$ has no Galois conjugate on the unit circle.
For simplicity, we assume that $\nu_0=(\delta_{-1}+\delta_{1})/2$.

We first recall a simple counting lemma.

\begin{lemma}\label{lem:counting}
Let $\l$ be an algebraic unit and denote by $k$ the number of Galois conjugates of $\l$ on the unit circle.
Then $|\Supp(\mu_\l^{(\ell)})|\le C \ell^kM_\l^\ell$, where $C$ is a constant depending only on $\l$. In particular $$\rho_\l \le \min\{1,\log M_\l\}.$$
\end{lemma}
Recall that $\rho_\l$ is defined in \eqref{roelldef}. This lemma is standard, but we give the proof for the reader's convenience.
\begin{proof}
We denote by $\s_1,\ldots \s_n:\Q(\l)\to\R$ the real Galois embeddings with $|\s_i(\l)|\ge 1$
and by $\tau_1,\ldots,\tau_m:\Q(\l)\to\C$ the complex Galois embeddings with $|\tau_i(\l)|\ge 1$
such that we take exactly one from each pair of complex conjugate embeddings.
Furthermore, we denote by $\rho_1,\ldots,\rho_o:\Q(\l)\to\C$ the real or complex Galois embeddings with $|\rho_i(\l)|< 1$.
(Here we take both from a pair of complex conjugate embeddings.)
We define the map $S:\Q(\l)\to \R^{n+2m}$ by
\[
S(x)=(\s_1(x),\ldots,\s_n(x),\Re(\tau_1(x)),\Im(\tau_1(x)),\ldots,\Re(\tau_m(x)),\Im(\tau_m(x))).
\]

We consider the set
\[
A=\Supp(\mu_\l^{(\ell)})=\Big\{\sum_{i=0}^{\ell-1} a_i\l^i:a_i\in\{-1,1\};\text{for all $i$}\Big\}.
\]
We note that elements of $A$ are algebraic integers, hence for any two different $x,y\in A$ we have
\[
\prod_{i,j,k}|\s_i(x-y)||\tau_j(x-y)|^2|\rho_k(x-y)|\ge 1.
\]
For any $1\le k\le o$, we have $|\rho_k(x-y)|\le |\rho_k(x)|+|\rho_k(y)|\le2/(1-|\rho_k(\l)|)$.
Hence there is a number $c$ depending on $\l$
such that
\[
\prod_{i,j}|\s_i(x-y)||\tau_j(x-y)|^2\ge c.
\]
Thus $\|S(x-y)\|>c_1$ for some other number $c_1$ depending on $\l$.

Consider the set
\begin{align*}
\Omega=\{(x_1,&\ldots,x_{n+2m})\in\R^{n+2m}:\\
&|x_i|\le \frac{|\s_i(\l)|^{\ell}-1}{|\s_i(\l)|-1}+c_1\;\text{for $1\le i\le n$, while for $1\le j\le m$}, \\
&|x_{n+2j-1}|,|x_{n+2j}|\le \frac{|\tau_j(\l)|^{\ell}-1}{|\tau_j(\l)|-1}+c_1\;\text{if $|\tau_j(\l)|>1$}, \\
&|x_{n+2j-1}|,|x_{n+2j}|\le \ell+c_1\;\text{if $|\tau_j(\l)|=1$}
\}.
\end{align*}
It is easily seen that the balls of radii $c_1/2$ around the points $S(x)$ for $x\in A$ are disjoint and contained
in $\Omega$.
On the other hand, there is a number $C$ depending only on $\l$ such that the volume of $\Omega$ is less than
$C\ell^kM_\l^\ell$, hence the claim follows.
\end{proof}

\begin{proof}[Proof of the upper bound in Theorem \ref{main}] Using \eqref{support} we see that $h_\l \leq \rho_\l$, so the bound follows from Lemma
 \ref{lem:counting}.
\end{proof}

It remains to prove Proposition \ref{prp:strict}, which we recall now:
\begin{proposition}\label{prp:strict1}
Let $\l$ be an algebraic number such that $M_\l<2$ and assume that $\l$ has no conjugates
on the unit circle.
Then $h_\l<\log M_\l$.
\end{proposition}

The proof is based on ideas from Garsia's proof that $h_\l<-\log\l$ if $\l^{-1}$ is Pisot \cite{Garsia-entropy}.

We use the following notation (note that it differs from that of the proof of Lemma \ref{lem:counting}).
We denote by $\s_1,\ldots \s_n:\Q(\l)\to\R$ the real Galois embeddings with $|\s_i(\l)|< 1$
and by $\tau_1,\ldots,\tau_m:\Q(\l)\to\C$ the complex Galois embeddings with $|\tau_i(\l)|< 1$
such that we take exactly one from each pair of complex conjugate embeddings.
Furthermore, we denote by $\rho_1,\ldots,\rho_o:\Q(\l)\to\C$ the real or complex Galois embeddings with $|\rho_i(\l)|> 1$.
(Here we take both from a pair of complex conjugate embeddings.)
We define the map $S:\Q(\l)\to \R^{n+2m}$ by
\[
S(x)=(\s_1(x),\ldots,\s_n(x),\Re(\tau_1(x)),\Im(\tau_1(x)),\ldots,\Re(\tau_m(x)),\Im(\tau_m(x))).
\]

We introduce the random vectors
\[
Y_\l=\sum_{i=0}^{\infty}\xi_i S(\l^i)=S(X_\l),\quad
Y_\l^{(\ell)}=\sum_{i=0}^{\ell-1}\xi_i S(\l^i)=S(X_\l^{(\ell)}).
\]
The strategy of the proof of the proposition is the following.
We begin by proving that the law of $Y_\l$ is singular (Lemma \ref{lem:singular}).
Then we approximate the law of $Y_\l^{(\ell)}$ by $Y_\l$ and conclude that
most of the probability mass is concentrated on an $\e$ proportion of the
atoms.
This yields a slight improvement over the proof of the upper bound in Theorem \ref{main},
which is just enough to conclude that $h_\l<M_\l$.

\begin{lemma}\label{lem:singular}
Suppose that $\l$ is an algebraic unit that has no Galois conjugates on the unit circle.
Then the law of $Y_\l$ is singular.
\end{lemma}

The proof is a straightforward generalization of the fact that $\mu_\l$ is singular if $\l^{-1}$ is Pisot.

\begin{proof}
We put for each $t \in \N$
\begin{align*}
\zeta_t=&(\s_1(\l^{-t}),\ldots,\s_n(\l^{-t}),
2\Re(\tau_1(\l^{-t})),-2\Im(\tau_1(\l^{-t})),\ldots,\\&2\Re(\tau_m(\l^{-t})),-2\Im(\tau_m(\l^{-t}))).
\end{align*}
We show below that there is a number $c>0$ depending only on $\l$ such that
\[
|\E[\exp(2\pi i\langle \zeta_t,Y_\l\rangle)]|\ge c.
\]
Then the Riemann-Lebesgue Lemma implies that the law of $Y_\l$ is not absolutely continuous.
Since the law is self-affine, it is of pure type, hence it is singular.

We observe that
\begin{align*}
\E[\exp(2\pi i\langle \zeta_t,Y_\l\rangle)]=&
\prod_{j=0}^\infty\E[\exp(2\pi i\langle\zeta_t,\xi_j S(\l^j)\rangle)]\\
=&\prod_{j=0}^\infty\cos(2\pi \langle\zeta_t, S(\l^j)\rangle).
\end{align*}
We can write
\begin{align*}
\langle\zeta_t, S(\l^j)\rangle=&\sum_{a=1}^n\s_a(\l^{-t})\s_a(\l^j)\\
&+2\sum_{b=1}^{m}(\Re(\tau_b(\l^{-t}))\Re(\tau_b(\l^{j}))-\Im(\tau_b(\l^{-t}))\Im(\tau_b(\l^{j})))\\
=&\sum_{a=1}^n\s_a(\l^{j-t})+2\sum_{b=1}^m\Re(\tau_b(\l^{j-t})).
\end{align*}

We set
\begin{align*}
u_j&=\sum_{a=1}^n\s_a(\l^{j})+2\sum_{b=1}^m\Re(\tau_b(\l^{j})),\\
v_j&=\sum_{a=1}^o\rho_a(\l^{j}).
\end{align*}
Since $\l$ is a unit and $u_j+v_j$ equals to the sum of all Galois conjugates of $\l^j$,
it follows that $u_j+v_j$ is an integer for all $j\in\Z$.
Hence we can write
\begin{align*}
|\E[\exp(2\pi i\langle \zeta_t,Y_\l\rangle)]|
\ge& \prod_{j=-\infty}^\infty|\cos(2\pi u_j)|\\
=&  \prod_{j=-\infty}^{-1}|\cos(2\pi v_j)|\prod_{j=0}^{\infty}|\cos(2\pi u_j)|=:c.
\end{align*}

We now show that the quantity of the right hand side, which we have denoted by $c$ is strictly positive.
We note that there is a positive number $\a<1$ depending only on $\l$ such that $|u_j|\le(n+2m)\a^j$ and $|v_j|\le o\a^{-j}$
for all $j\in\Z$.
We choose an integer $N$ large enough so that $(n+2m)\a^j<1/4$ and $o\a^j<1/4$ for $j>N$.
There is an absolute constant $C>0$ such that $\cos(x)\ge\exp(-C|x|)$ for $|x|<1/4$, hence we can write
\[
c\ge \exp\Big(-Co\sum_{j=-\infty}^{-N-1}\a^{-j}\Big)\prod_{j=-N}^N|\cos(2\pi u_j)|
\exp\Big(-C(n+2m)\sum_{j=N+1}^{\infty}\a^{j}\Big).
\]
We note that $u_j$ is an algebraic integer for all $j\in\Z$ so it cannot be equal to one half plus an integer.
This finishes the proof that $c>0$, hence the lemma follows.
\end{proof}

\begin{proof}[Proof of Proposition \ref{prp:strict1}]
Denote by $Q_\ell$ the box in $\R^{n+2m}$ with side lengths
\[
|\s_1(\l)|^\ell,\ldots,|\s_n(\l)|^\ell,
|\tau_1(\l)|^\ell,|\tau_1(\l)|^\ell,|\tau_2(\l)|^\ell,|\tau_2(\l)|^\ell\ldots,|\tau_m(\l)|^\ell,|\tau_m(\l)|^\ell
\]
centered around the origin.

We begin by an observation about the separation of points in the support of $Y_\l^{(\ell)}$.
Let $x,y\in \Supp X_\l^{(\ell)}$ be two different points.
We can apply an argument similar to that in the proof of Lemma \ref{lem:counting}
for the points $S(\l^{-(\ell-1)}x)$, $S(\l^{-(\ell-1)}y)$ and $\l^{-1}$ in place of $\l$ to show that
there is a number $c>0$ depending only on $\l$ such that $S(\l^{-(\ell-1)}x)-S(\l^{-(\ell-1)}y)\notin c Q_0$.
By $cQ_0$ and by similar notation, we mean the dilation or contraction of $Q_0$ by the factor $c$.
This in turn yields $S(x)-S(y)\notin c Q_{\ell-1}$.

Next, we estimate the difference between $Y_\l$ and $Y_\l^{(\ell)}$.
We can write
\[
|\s_i(X_\l-X_\l^{(\ell)})|=\Big|\sum_{j=\ell}^\infty\s_i(\l^j)\Big|\le C |\s_i(\l)|^{\ell},
\]
where $C$ is a constant depending only on $\l$.
A similar inequality holds for the embeddings $\tau_i$.
We can conclude hence that $Y_\l-Y_\l^{(\ell)}\in CQ_{\ell}$.

We fix a small number $\e>0$.
Since the law of $Y_\l$ is singular, we can find a closed set $A\subset \R^{n+2m}$ such that
$\mu(A)<\e$ and $\P(Y_\l\in A)>1-\e$.
Here and everywhere below, $\mu$ denotes the Lebesgue measure.

If $\ell$ is sufficiently large, we have $\mu(A+(C+\frac{c}{2})Q_{\ell})<2\e$,
where $c$ and $C$ are the same as above, because $A$ is closed.
We estimate the cardinality of $A':=\Supp(Y_\l^{(\ell)})\cap (A+CQ_\ell)$.
If $x,y\in A'$ are distinct, then $x+(c/2)Q_{\ell}$ and $y+(c/2)Q_{\ell}$ are disjoint.
Hence
\[
|A'|=\frac{\mu(A'+(c/2)Q_{\ell})}{\mu((c/2)Q_{\ell})}<C_1\e M_\l^{\ell},
\]
where $C_1$ is a constant depending only on $\l$.
By Lemma \ref{lem:counting} (here $k=0$), we can estimate from above the cardinality of $B':=\Supp(Y_\l^{(\ell)})\setminus A'$ by $C_2M_\l^{\ell}$. Moreover $\P(Y_\l^{(\ell)} \in B') \leq \P(Y_\l \notin A) < \e$.

We put our estimates together to bound $H(Y_\l^{(\ell)})$.
We write
\begin{align*}
H(Y_\l^{(\ell)})
=&\sum_{x\in A'}-\P(Y_\l^{(\ell)}=x)\log (\P(Y_\l^{(\ell)}=x))\\
&+\sum_{x\in B'}-\P(Y_\l^{(\ell)}=x)\log (\P(Y_\l^{(\ell)}=x))\\
\le &|A'|\cdot\Big(-\frac{\P(Y_\l^{(\ell)}\in A')}{|A'|}\log\Big(\frac{\P(Y_\l^{(\ell)}\in A')}{|A'|}\Big)\Big)\\
&+|B'|\cdot\Big(-\frac{\P(Y_\l^{(\ell)}\in B')}{|B'|}\log\Big(\frac{\P(Y_\l^{(\ell)}\in B')}{|B'|}\Big)\Big)\\
=&-\P(Y_\l^{(\ell)}\in A')\log(\P(Y_\l^{(\ell)}\in A'))-\P(Y_\l^{(\ell)}\in B')\log(\P(Y_\l^{(\ell)}\in B'))\\
&+\P(Y_\l^{(\ell)}\in A')\log|A'|+\P(Y_\l^{(\ell)}\in B')\log|B'|\\
\le& 1+(1-p)\log(C_1\e M_{\l}^{\ell})+p\log(C_2 M_{\l}^{\ell})
= \log(2C_1\e M_{\l}^{\ell}) - p \log(\frac{C_1}{C_2}\e),
\end{align*}
where $p=\P(Y_\l^{(\ell)}\in B')\leq \e$.
The inequality in the second line follows from the concavity of the function $F(x)=-x \log(x)$.
If we set $\e$ sufficiently small depending only on $C_1$ and $C_2$, hence ultimately depending only on $\l$,
then we obtain $H(Y_\l^{(\ell)})<\log( M_{\l}^{\ell})$.
We can conclude now that
\[
h_\l\le \frac{H(Y_\l^{(\ell)})}{\ell}<\log(M_\l)
\]
proving the claim.
\end{proof}

\subsection{}\label{sc:hochman}

Although \eqref{hoch} is not stated in this form in \cite{hochman} it is essentially contained in that paper.
For the reader's convenience, we show how to reduce it
to the main result of \cite{hochman} that we now recall.
All of the ideas in this paragraph are taken from \cite{hochman}.

For an integer $n$ denote by $\cD_n$ the partition of $\R$ into intervals of length $2^{-n}$
such that $0$ is an endpoint of two intervals in the partition and by  $H(\nu,\cD_n)$ the Shannon entropy
of $\nu$ with respect to the partition $\cD_n$.
For integers $m<n$, denote by $H(\nu,\cD_n|\cD_m)=H(\nu,\cD_n)-H(\nu,\cD_m)$ the conditional entropies.
This is the notation in \cite{hochman}, which differs from ours, and we only use it in this paragraph.

\begin{theorem}[Special case of \cite{hochman}*{Theorem 1.3}]
Let $\l\in(1/2,1)$ and suppose that $\dim{\mu_\l}<1$.
Then
\begin{equation}\label{eq:th-hoch}
\lim_{n\to\infty} \frac{1}{n'}H(\mu_\l^{(n)},\cD_{qn'}|\cD_{n'})=0
\end{equation}
for any $q>0$, where $n'=\lfloor n\cdot \log\l^{-1}\rfloor$.
\end{theorem}

We first consider the case when $\dim{\mu_\l}<1$.
We recall from \cite{Garsia-arithmetic}*{Lemma 1.52} that for each algebraic number $\l$,
there are $c_\l$ and $d_\l$ such that any two distinct points in the support of
$\mu_\l^{(n)}$ are of distance at least $c_\l n^{-d_\l}M_\l^{-n}$.
Hence, taking $q$ sufficiently large, we can write
\[
\lim_{n\to\infty} \frac{1}{n}H(\mu_\l^{(n)},\cD_{qn'})=\lim_{n\to\infty} \frac{1}{n}H(\mu_\l^{(n)})=h_\l.
\]

We combine this with \eqref{eq:th-hoch} and deduce
\[
\lim_{n\to\infty} \frac{1}{n'}H(\mu_\l^{(n)},\cD_{n'})=\frac{h_\l}{\log\l^{-1}}.
\]
Since $\mu_\l^{(n)}$ approximates $\mu_\l$ at scale $2^{-n'}$,
\[
\lim_{n\to\infty} \frac{1}{n'}H(\mu_\l,\cD_{n'})=\lim_{n\to\infty} \frac{1}{n'}H(\mu_\l^{(n)},\cD_{n'})
\]
The quantity on the left hand side of this equation is known to equal $\dim \mu_\l$, see
\cite{Feng-Hu}.

It is left to consider the case $\dim{\mu_\l}=1$, and we need to prove that $h_\l\ge \log\l^{-1}$.
The observation that $h_\l<\log \l^{-1}$ implies $\mu_\l$ is singular goes back to Garsia \cite{Garsia-entropy}.
Here we need a slightly stronger statement that we obtain by writing
\begin{align*}
1=&\dim\mu_\l=\lim_{n\to\infty} \frac{1}{n'}H(\mu_\l,\cD_{n'})=\lim_{n\to\infty} \frac{1}{n'}H(\mu_\l^{(n)},\cD_{n'})\\
\le &\lim_{n\to\infty} \frac{1}{n'}H(\mu_\l^{(n)})=\frac{h_\l}{\log \l^{-1}}.
\end{align*}
This completes the proof of \eqref{hoch}.

We can now prove Theorem \ref{lehmer}.

\begin{proof}[Proof of Theorem \ref{lehmer}]
If Lehmer's conjecture holds, then there is a number $\e>0$ such that
$\log M_\l>\e$ for all algebraic numbers $\l\in(1/2,1)$.
Then $h_\l\ge 0.44\e$ by Theorem \ref{main}.
If $\l$ is sufficiently close to $1$ so that $\log \l^{-1}<0.44\e$, then $h_\l>\log\l^{-1}$ and
$\dim\mu_\l=1$ by \eqref{hoch}.
\end{proof}

\section{Uniform exponential growth for linear groups} \label{uniform}
This section is devoted to the proof of Theorem \ref{equi} and the consequences of Theorem \ref{main} for group growth. Recall that given a group $G$ generated by a finite subset $S$, we denote by $\rho_S$ the rate of exponential growth:

$$\rho_S := \lim_{n\to +\infty} \frac{1}{n} \log |S^n|.$$

Since $|S^{n+m}| \leq |S^n||S^m|$ for every $n,m \in \N$ the above limit exists, by the classical subadditive lemma. We also note that $\rho_S=\rho_{S \cup \{1\}}$, because $(S \cup \{1\})^n=\bigcup_{i=0}^n S^i$, and hence $|(S \cup \{1\})^n| \leq (n+1)|S^n|$. So without loss of generality, we will assume that $1 \in S$.

Before going into any details we record here the following initial observation, whose proof we leave to the reader. Let $g_{\l,1}$ and $g_{\l,-1}$ be the affine transformations of the complex line $x \mapsto \l x+1$ and $x \mapsto \l x -1$. Set $S_\l := \{g_{\l,1}, g_{\l,-1}\}$. Then we have:
$$\rho_{S_\l} = \rho_\l,$$
where $\rho_\l$ is defined in $(\ref{roelldef})$. In particular Theorem \ref{main} combined with $(\ref{uppp})$ implies that $\rho_{S_\l} \geq 0.44 \min\{1,\log M_\l\}$.

Recall that an abstract group is said to satisfy a certain property $\mathcal{P}$ \emph{virtually}, or equivalently to be \emph{virtually $\mathcal{P}$}, if it contains a subgroup of finite index with the said property $\mathcal{P}$. For example a group is virtually trivial if and only if it is finite.

Recall further that a group $G$ is said to be \emph{solvable} if the derived series of the group stabilises to the trivial group in finitely many steps, namely setting $G_1=G$ and recursively $G_{n+1}:=[G_n,G_n]$  the subgroup generated by all commutators $aba^{-1}b^{-1}$, $a,b \in G_n$ there is $n<\infty$ such that $G_n=\{1\}$. Similarly a group is said to  be \emph{nilpotent} if the central descending series stabilises to the trivial group in finitely many steps,  namely setting $G^{(1)}=G$ and recursively $G^{(n+1)}:=[G,G^{(n)}]$  the subgroup generated by all commutators $aba^{-1}b^{-1}$, $a\in G$ and $b \in G_n$ there is $n<\infty$ such that $G^{(n)}=\{1\}$. Examples of solvable groups include the group $\Upp_d(\C)$ of upper triangular invertible matrices of size $d$. Examples of nilpotent groups include the commutator subgroup of $\Upp_d(\C)$, i.e. the upper triangular and unipotent matrices (i.e. matrices all of whose eigenvalues are $1$).

According to a celebrated lemma of Jordan \cite{jordan}, there is a function $J=J(d) \in \N$ such that every finite subgroup of $\GL_d(\C)$ contains a normal abelian subgroup of index at most $J(d)$. For this and for general background on linear groups we refer the reader to standard books \cites{dixon, raghunathan, wehrfritz}. 

The following is the main theorem of this section:

\begin{theorem}\label{solgrowth} Let $S$ be a finite subset of $\GL_d(\C)$ generating a virtually solvable subgroup, then either $\rho_S=0$ and $\langle S \rangle$ is virtually nilpotent, or there is $\lambda \in \C^\times$, not a root of unity, such that
$$\rho_S \geq\frac{1}{27d!J(d)} \log  M_\lambda.$$
\end{theorem}

Recall that $M_\lambda$ is the Mahler measure of the minimal polynomial of $\lambda$ in $\Z[X]$ if $\lambda$ is algebraic over $\Q$.
We adopt the convention that $M_\lambda=2$ if $\lambda$ is transcendental. We thus obtain the following consequence.

\begin{corollary} Assuming the Lehmer conjecture, there is $c>0$ such that the following holds. If  the finite subset $S \subset \GL_d(\C)$ generates a virtually solvable subgroup, then either $\rho_S=0$ and $\langle S \rangle$ is virtually nilpotent, or $$\rho_S>\frac{c}{d!J(d)}.$$
\end{corollary}

This completes the proof of the equivalence between the Lehmer conjecture and the Growth conjecture (Theorem \ref{equi}), see the Introduction.

\begin{remark} \label{jordan-bound}
The classical geometric proofs of Jordan's lemma by Bieberbach and Frobenius give a bound on $J(d)$ of the form $d^{O(d/\log d)^2}$, while using the classification of finite simple groups much better bounds have been obtained by B. Weisfeiler and then by M. Collins (see \cite{collins}) who shows the sharp bound $J(d) \leq (d+1)!$ when $d \geq 71$.
\end{remark}

\begin{remark} It is worth remarking here that the standard argument for proving exponential word growth in finitely generated groups is to exhibit two elements that generate a free semi-group. Every non virtually nilpotent solvable group contains a free semi-group (see e.g. \cite{breuillard-solvable} and references therein). However there is no uniform bound on the word length of these free generators: indeed in \cite{breuillard-solvable} a construction is given of a sequence of algebraic numbers $\l_n$ such that $S_{\l_n}$ generates a non-virtually nilpotent subgroup of affine transformations, and yet no pair of elements in $(S_{\l_n})^n$ generate a free semi-group. Therefore there is no hope of obtaining a good lower estimate for $\rho_S$ using ping-pong techniques only, as was done in \cites{breuillard-tits,breuillard-zimmer} in the non-virtually solvable case. Instead Theorem \ref{solgrowth} will be a consequence of the entropy lower bounds established in the first part of this paper. \end{remark}

We will show additionally, that if $K_S$ denotes the field generated by the matrix entries of each element $s \in S$, then $\lambda$ can be chosen to be algebraic over $K_S$ of degree at most $d!$. Using Dobrowolski's bound for the Mahler measure of an algebraic number of bounded degree (see \cite{dobrowolski}) we obtain in a similar way the following consequence, which was pointed out to us by Andreas Thom.

\begin{corollary}\label{corq} There is an absolute constant $c>0$ such that if the finite subset $S \subset \GL_d(\Q)$ generates a virtually solvable subgroup, then either $\rho_S=0$ and $\langle S \rangle$ is virtually nilpotent, or $$\rho_S>\frac{c}{d^{2d}}.$$
\end{corollary}

In the proof of Theorem \ref{solgrowth}, we first show that $\Gamma$ contains a finite index subgroup $H$ that
can be conjugated into $\Upp_d(\C)$,  the subgroup of $\GL_d(\C)$ made of upper triangular matrices.
Then we show that $H$ has a non-virtually solvable image under a suitable homomorphism into
$\Aff(\C)$, the group of affine transformations of the complex line (similarities).
This allows us to reduce the theorem to the special case of $S\subset \Aff(\C)$, which we treat first.

Note that we have an isomorphism:

$$\Aff(\C) \simeq \{ g_{a,b}:= \left(
                      \begin{array}{cc}
                        a & b \\
                        0 & 1 \\
                      \end{array}
                    \right) ; a \in \C^* , b \in \C\} \leq \GL_2(\C)$$

The matrix $g_{a,b}$ identifies with the affine transformation $x \mapsto ax+b$. If $g=g_{a,b}$, we set $a(g)=a$ and $b(g)=b$.

\begin{lemma} Let $S$ be a finite subset of $\Aff(\C)$ containing the identity. If $\langle S \rangle$ is not virtually nilpotent, then there is $g,\gamma \in S^3$ such that the pair $(g,\gamma)$ is conjugate in $\Aff(\C)$ to the pair $(g_{\lambda,1}, g_{\lambda,-1})$ for some $\lambda \in \C^*$, which is not a root of unity.
\end{lemma}

\begin{proof}
The proof relies on the following fact, whose proof we leave to the reader.
If $\lambda \in \C\setminus \{0,1\}$ and $t_1 \neq t_2$, $u_1 \neq u_2$, then the pair $(g_{\lambda,t_1}, g_{\lambda,t_2})$ is conjugate to the pair $(g_{\lambda,u_1}, g_{\lambda,u_2})$ by an element in $\Aff(\C)$.

First observe that the multiplicative subgroup of $\C^*$ generated by the $a(s)$, $s \in S$, is infinite, for otherwise the subgroup $\langle S \rangle$ would be virtually abelian. In particular, there is $s_0 \in S$ such that $\lambda_0:=a(s_0)$ is not a root of unity.
Up to conjugating $S$ in $\Aff(\C)$, we may assume without loss of generality that $b(s_0)=0$. So $s_0=g_{\lambda_0,0}$.

Now note that there must exist some $s \in S$ such that $b(s) \neq 0$, for otherwise $\langle S \rangle$ would be abelian. Now not both $\lambda_0 a(s)$ and $\lambda_0^2 a(s)$ are roots of unity. Let $\gamma$ be either $s_0s$ or $s_0^2s$, so that $a(\gamma)$ is not a root of unity. Accordingly, let $g$ be either $ss_0$, or $ss_0^2$.

Then $\lambda:=a(\gamma)=a(g)$ is not a root of unity, while $b(\gamma) \neq b(g)$. From the above fact, we deduce that $(g,\gamma)$ is conjugate in $\Aff(\C)$ to the pair $(g_{\lambda,1}, g_{\lambda,-1})$ as desired.
\end{proof}

If $S \subset \GL_d(\C)$, recall that $K_S$ denotes the subfield of $\C$ generated by the matrix entries of each $s \in S$.

\begin{corollary}\label{affcor} Let $S$ be a finite subset of $\Aff(\C)$ containing the identity. Assume that $\langle S \rangle$ is not virtually nilpotent, then there is  $\lambda \in K_S \setminus \{0\}$, which is not a root of unity, such that
$$\rho_S \geq \frac{1}{9} \log M_\lambda.$$
\end{corollary}

\begin{proof} After replacing $S$ by a conjugate, $S^{3n}$ contains all products of at most $n$ affine transformations of the form $g_{\lambda, 1}$ or $g_{\lambda,-1}$ for some $\l\in K_S$. The images of $0$ under these maps  form the support of the measure $\mu_\l^{(n)}$, so $\rho_S \geq \frac{1}{3} \rho_\lambda$. The desired inequality then follows from Theorem \ref{main}. If $M_\l \geq 2$, set $\l=2$.
\end{proof}

In what follows will make use of certain notions from the theory of linear algebraic groups for which we refer to the textbook  \cite{humphreys}. For the reader's convenience we briefly review some of the terminology. 

A subset of $M_n(\C)$ is said to be Zariski-closed if it is the set of zeroes of a family of polynomials in the matrix entries. For example $\GL_n(\C)$ is viewed as the Zariski closed subset of $M_{n+1}(\C)$ of bloc diagonal matrices $\diag(A,x)$, $A \in \GL_n(\C)$, $x \in \C$ such that $\det(A)x=1$. This endows $\GL_n(\C)$ with a non-Hausdorff topology called the Zariski topology.

The Zariski closure of a subset is the smallest (i.e. the intersection of all the) Zariski closed subset containing it. A Zariski closed subset is called irreducible if it is not the union of two proper Zariski closed subsets. Every Zariski closed subset is the union of finitely many irreducible Zariski closed subsets called its irreducible compoments. There is a well-defined notion of dimension of a Zariski-closed subset. Zariski closed subgroups are closed complex Lie subgroups of $\GL_n(\C)$ and their Zariski dimension coincides with the complex dimension of their Lie algebras.

The irreducible components of a Zariski closed subgroup $G\leq \GL_n(C)$ are disjoint: they are the cosets of the unique irreducible component containing the identity, called the connected component of the identity and denoted by $G^\circ$. A Zariski closed subgroup $G$ is said to be connected if $G=G^\circ$. The Zariski closure of a subgroup (or sub-semi-group) of $\GL_n(\C)$ is a group.

A unipotent subgroup is a subgroup made entirely of unipotent elements. Since we are in characteristic zero, every Zariski closed unipotent subgroup is connected (\cite{humphreys}*{p. 101}). The union of all Zariski closed unipotent normal subgroups of a Zariski closed subgroup $G \leq \GL_n(\C)$ is itself a Zariski closed unipotent normal subgroup, called the unipotent radical of $G$ and denoted by $G_u$.  A subgroup $H \leq \GL_n(\C)$ is said to be diagonalizable if it can be conjugated inside the subgroup of diagonal matrices. A Zariski closed and connected diagonalizable subgroup of $G$ of maximal dimension is called a maximal torus. Any two maximal tori are conjugate in $G$ (\cite{humphreys}*{\S 21}).

Having recalled this terminology we can now state a technical result about subgroups of  $\GL_n(\C)$
that can be conjugated into $\Upp_n(\C)$.
This will be used both for finding the finite index subgroup in $\Gamma$ that can be conjugated into $\Upp_n(\C)$
and for finding the non-virtually solvable homomorphic image of that group in $\Aff(\C)$.

\begin{lemma}\label{solvable}Let $G$ be a solvable Zariski closed algebraic subgroup of $\GL_n(\C)$, and let $G^\circ$ be the connected component of the identity. The following are equivalent:
\begin{enumerate}
\item $[G,G]$ is unipotent,
\item $G$ is a subgroup of a connected solvable algebraic subgroup of $\GL_n(\C)$,
\item $G$ can be conjugated into $\Upp_n(\C)$.
\end{enumerate}
Moreover if this holds, then there is a finite abelian subgroup $F$ such that $G=F G^\circ$, and a diagonalizable subgroup $H \leq G$ containing $F$ such that $H=FH^\circ$,  and $G^\circ =H^\circ  G_u$ where $G_u$ is the unipotent radical of $G^\circ$.
\end{lemma}

\begin{proof}
The equivalence of $(2)$ and $(3)$ is the content of  the Lie-Kolchin theorem, see \cite{humphreys}*{17.6},
and $(1)$ trivially follows from $(3)$.

In order to show that $(1)$ implies $(3)$, we first show that the groups $H$ and $F$ with the properties stated in the lemma exist.
First recall that it is a well-known observation attributed to Platonov \cite{platonov} (see also \cite{wehrfritz}*{10.10}, \cite{borel-serre}*{5.11}) that every complex linear algebraic group has a finite subgroup intersecting each irreducible component. So there is a finite subgroup $F$ with $G=F G^\circ$. Assuming $(1)$ and the existence of $F$ such that $G=FG^\circ$, we will prove the existence of $H$ as above by induction on $\dim G$. 

If $(1)$ holds then $[G,G]$ is contained in $G_u$ (its Zariski closure is a closed normal unipotent subgroup of $G$). This implies that $F$ must be abelian:   $[F,F]$ is finite and unipotent, hence trivial. Being finite, $F$ must consist of semisimple (i.e. diagonalizable) elements, and since it is abelian, it  is a diagonalizable subgroup.

To find $H$, we argue as in the standard proof of the existence of a maximal torus mapping onto the quotient of a connected solvable algebraic group with its unipotent radical (\cite{humphreys}*{19.3}). Let $Z_G(F)$ be the centralizer of $F$ in $G$ and $Z_G(F)^\circ$ its connected component of the identity. Since $[F,G]\subset G_u$, the diagonalizable subgroup $F$ acts trivially by conjugation on $G/G_u$, and \cite{humphreys}*{Corollary 18.4} tells us that the map $\phi:G \to G/G_u$ sends $Z_G(F)^\circ$ onto $G^\circ/G_u$.

It follows that  $G=FG^\circ = F Z_G(F)^\circ G_u$, so  $Z_G(F)=FZ_G(F)^\circ (G_u \cap Z_G(F))$. But Zariski closed unipotent subgroups are connected, so $G_u \cap Z_G(F) \leq Z_G(F)^\circ$ and we conclude that $Z_G(F)=FZ_G(F)^\circ$.  So we may replace $G$ by $Z_G(F)$ and argue by induction if $\dim Z_G(F)< \dim G$. The subgroup $H$ thus found for $Z_G(F)$ will work for $G$ as well. Otherwise $G^\circ \leq Z_G(F)$, and given a maximal torus in $H^\circ \leq G^\circ$, we may set $H:=FH^\circ$, which is the desired diagonalizable subgroup. Since $G^\circ$ is connected and solvable we have $G^\circ=H^\circ G_u$ by \cite{humphreys}*{\S 19.3}.

Therefore, $H$ and $F$ exist as in the statement of the lemma. Since $G_u$ is unipotent and normal (even characteristic)  in $G$, its fixed point subspace is non-trivial and $G$-invariant. In fact there is a flag of $G$-invariant subspaces such that $G_u$ acts trivially on each successive quotient. In particular $H$ preserves this flag and it and can thus be diagonalized in an adapted basis. In that basis $G=HG_u$ is upper triangular, and $(3)$ follows. 
\end{proof}

\begin{remark}\label{rati}
Denote by $K$ the field of definition of the algebraic group $G$ that appears in the previous lemma.
It is not a priori clear that the subgroup $F$ in the conclusion of
the lemma can be chosen inside $G(K)$.
However, we show now that this is indeed the case if $G \leq \Upp_n$, and moreover $H$ can be chosen defined and split over $K$. This observation is not needed for the proof of the main result Theorem \ref{solgrowth}, but it will be used for Corollary \ref{corq}.

To see it, we need to go back to the proof of Platonov's obervation, as given for example in \cite{wehrfritz}*{\S 10}. We first reduce to the case when $G^\circ$ is nilpotent. Note that since maximal $K$-tori in $\Upp_n$ are $K$-split (i.e. isomorphic to $(\C^{\times})^n$ via an isomorphism defined over $K$), so are the maximal $K$-tori $T$ in $G^\circ$.
Since they are all conjugate by an element of $G_u(K)$, we have $G(K)= N_G(T)(K) G_u(K)$. But $N_G(T)=Z_G(T)$ (apply \cite{humphreys}*{19.4.b} with $G=\Upp_n$), so replacing $G$ with $Z_G(T)$, we may assume that maximal $K$-tori in $G^\circ$ are $K$-split and are inside the center of $G$. This implies in particular that there is only one such $T$ and $G^\circ=TG_u$ is nilpotent.

Then $G/T$ is virtually unipotent, and $G^\circ(K)/T(K)$ is divisible, torsion free and nilpotent. Now it follows, as in the discussion \cite{wehrfritz}*{\S 10.10}, that the exact sequence $1 \to G_u(K) \to G(K)/T(K) \to G(K)/G^\circ(K) \to 1$  splits.

So there is a subgroup $F \subset  G(K)$ such that $F \cap G_u(K) = 1$ (forcing $F$ to be abelian) and $G(K)=FT(K)G_u(K)$. Clearly $F$ is made of semisimple elements (if $f \in F$ some power of $f$ lies in $T$). Also every semisimple element in $\Upp_d(K)$ is diagonalizable. It follows that $FT(K)$ is an abelian subgroup made of diagonalizable elements, hence it can be simultaneously diagonalized, and this yield the desired subgroup $H$.
\end{remark}

We now move towards the proof of Theorem \ref{solgrowth}. We will use the previous lemma to reduce to the case of the $2$-dimensional affine group. Crucial to this reduction is the following

\begin{lemma}\label{mapaffine} Let $\Gamma$ be a subgroup of $\Upp_d(\C)$. If $\Gamma$ is not virtually nilpotent, then there is a homomorphism $\rho: \Gamma \to \Aff(\C)$, whose image is not virtually nilpotent.
\end{lemma}

\begin{proof} We prove the result for subgroups $\Gamma$ of a connected solvable algebraic group $\G$ in place of $\Upp_d(\C)$. (This is equivalent to our assumption by the Lie-Kolchin theorem  \cite{humphreys}*{17.6}). We will work by induction on $\dim \G$. Without loss of generality (passing to the Zariski closure $G$ of $\Gamma$), we may assume that $\Gamma=G$ is Zariski-closed, because if $\rho(\Gamma)$ is virtually nilpotent so will be $\rho(G)$. Under the assumption that $G$ is Zariski-connected, a proof of this lemma can be found in \cite{bg-toti}*{Lemma 10.7}. We need some adjustments to handle the general case.

 To prove Lemma \ref{mapaffine}, we have to find a character $\chi: G \to \C^*$, and a non-trivial cocycle $\beta: G \to \C$, that is a map such that $\beta(gh)=\beta(g)+\chi(g)\beta(h)$, and $\beta(\ker \chi)\neq 0$. Then the map $G \to \Aff(\C)$ sending $g$ to the matrix $$\left(
                                                                                                                                   \begin{array}{cc}
                                                                                                                                     \chi(g) & \beta(g) \\
                                                                                                                                     0 & 1 \\
                                                                                                                                   \end{array}
                                                                                                                                 \right)$$ gives the desired homomorphism.

Since $[\G,\G]$ is unipotent, we can apply Lemma \ref{solvable}  above to $G$. Hence there is a diagonalizable subgroup $H$ of $G$ such that $G=HG_u$. The subgroup $H$ lies in a maximal torus of $\G$, say $T$, so that $\G=T\cdot U$, where $U$ is the unipotent radical of $\G$.  Note further that $G_u=G \cap U$. 

Let $Z$ be the center of $U$. It is a normal algebraic subgroup of $\G$ of positive dimension. If $G^\circ$ acts trivially on $Z$ by conjugation, then we may pass to $\G/Z$ and apply induction, since then the image of $G$ in $\G/Z$ will not be virtually nilpotent.

So assume that its action is not trivial. Note that the $G$-action on $Z$ factors through $\G/U \simeq T$. Since $T$ is a torus, its action on the additive group $Z \simeq \C^d$ splits into weight spaces. There is a weight $\chi: T \to \C^*$ such that $\chi(H^\circ) \neq 1$. Let $Z_\chi \leq Z$ be a one-dimensional subspace in the eigenspace of $\chi$, so that $tzt^{-1}=\chi(t)z$ when $t \in T$.

Note that $Z_\chi$ is a normal subgroup of $\G$.  We can assume that $H^\circ$ acts trivially on $U/Z_\chi$, for otherwise the image of $G$ in $\G/Z_\chi$ would not be virtually nilpotent and we could again use induction. This means that $[H^\circ,U] \leq Z_\chi$.

On the other hand $H^\circ$ does not commute with $G_u$, for otherwise $G^\circ$ would be nilpotent. Pick $h_0 \in H^\circ$ such that $[h_0,G_u] \neq 1$. Since $G=HG_u$, and $H$ commutes with $h_0$, we see that $[h_0,G] \leq Z_\chi$. Then we set $\beta(g)=[h_0,g]$ for $g \in G$, after identifying $Z_\chi$ with the additive group of $\C$. This yields the desired a non trivial cocycle as claimed and ends the proof of the lemma.
\end{proof}

\begin{remark}\label{ration} If $\G$ is a $K$-split connected solvable $K$-subgroup of $\GL_d(\C)$, where $K$ is some subfield of $\C$ and $G \leq \G$ is a closed algebraic $K$-subgroup, which is not virtually nilpotent, then replacing everywhere \emph{maximal torus} by \emph{$K$-split maximal torus}, the proof above combined with Remark \ref{rati} shows that the homomorphism $\rho: G \to \Aff(\C)$ we have constructed is defined over $K$.
\end{remark}

Recall $M_\lambda$ denotes the Mahler measure of $\lambda$ and that we have adopted the convention that $M_\lambda=2$ if $\lambda$ is transcendental. We conclude:

\begin{corollary} Let $S$ be a finite subset of $\Upp_d(\C)$ containing the identity. Assume that $\langle S \rangle$ is not virtually nilpotent, then there is  $\lambda \in \C^\times$, which is not a root of unity, such that
$$\rho_S \geq \frac{1}{9} \log M_\lambda.$$
\end{corollary}

\begin{proof} This follows from the combination of Lemma \ref{mapaffine} and Corollary \ref{affcor}.
\end{proof}

\begin{remark}\label{cor-rati} If $K_S$ denotes the field generated by the matrix entries of each element $s \in S$, then $\lambda$ in the previous statement can be found in $K_S\setminus \{0\}$. This follows from the same argument together with Remarks \ref{rati} and \ref{ration}.
\end{remark}

To handle virtually solvable subgroups not necessarily contained in $\Upp_d(\C)$, we need the following lemma.

\begin{lemma}\label{finiteindex} Let $G$ be a group, and $H$ a subgroup with $G=SH$ for some finite generating subset $S$ of $G$ containing $1$ (but not necessarily symmetric). Let $\rho: H \to \Aff(\C)$ be a homomorphism with non virtually nilpotent image. Then the subgroup generated by $S^3 \cap H$ has a non virtually nilpotent image under $\rho$.
\end{lemma}

\begin{proof}
Let $\chi$ be the character $H \to \C^*$ induced by the natural homomorphism $\Aff(\C) \to \C^*$. First we claim that $\chi(\langle S^3 \cap H \rangle)$ is infinite. Indeed $H$ is generated by the elements of the form $s_1s_2s_3^{-1}$ belonging to $H$, with each $s_i$ in $S$. At least one of them must map to an element of infinite order under $\chi$, say $\chi(s_1s_2s_3^{-1})$ has infinite order. Let $s_4 \in S$ be such that $s_3^{-1} \in s_4H$. Then $s_3s_4 \in H$. So we see that either $\chi(s_3s_4)$ has infinite order, or else $s_1s_2s_4=s_1s_2s_3^{-1}s_3s_4$ has infinite order. This proves the claim.

So pick $\gamma \in S^3 \cap H$ with $\chi(\gamma)$ of infinite order. Now observe that every virtually nilpotent subgroup of $\Aff(\C)$ containing $\rho(\gamma)$ must be abelian. Indeed if it contains an element not commuting with $\rho(\gamma)$, then the commutator will be a non-trivial translation $t$, but the subgroup generated by $\rho(\gamma)^n$ and $t^n$ is nilpotent for no $n \in \N$ (the centralizers of these two elements have trivial intersection, so there is no center).

So if $\rho(\langle S^3 \cap H \rangle)$ were virtually nilpotent, it would be abelian. However $\rho(H)$ is not abelian, and it is generated by the $\rho(s_1s_2s_3^{-1})$ with $s_1s_2s_3^{-1}$ belonging to $H$ and $s_1,s_2,s_3 \in S$. The centralizer of $\rho(\gamma)$ in $\Aff(\C)$ is abelian. Pick such elements with $[\rho(s_1s_2s_3^{-1}),\rho(\gamma)]\neq 1$. As above let $s_4 \in S$ with $s_3s_4 \in H$. We see that either $\rho(s_3s_4)$ does not commute with $\rho(\gamma)$, or else $\rho(s_1s_2s_4)$ does not commute with $\rho(\gamma)$. In both cases $\rho(\langle S^3 \cap H \rangle)$ is not abelian, hence not virtually nilpotent. The lemma is proved.
\end{proof}

We can now conclude the

\begin{proof}[Proof of Theorem \ref{solgrowth}] Let $G$ be the Zariski closure of the subgroup $\Gamma$ generated by $S$. Let $T$ be a maximal torus of $G^\circ$. For every $g \in G$, $gTg^{-1}$ is a maximal torus, hence is conjugate to $T$ by an element of $G^\circ$. This shows that $G=N_G(T)G^\circ$.

Diagonalizing $T$, we see that the centralizer $Z_G(T)$ of $T$ in $G$ has index at most $d!$ in $N_G(T)$.  By the result of Platonov mentioned at the beginning of the proof of Lemma \ref{solvable}, there is a finite subgroup $F$ such that $Z_G(T)=FZ_G(T)^\circ$, and Jordan's theorem implies that there is an abelian subgroup $A$ in $F$ with index at most $J(d)$.

Now note that $AG^\circ=ATG_u$, where $G_u$ is the unipotent radical of $G$, and $AT$ is abelian. Consequently  the commutator subgroup of the subgroup $AG^\circ$ is unipotent. So Lemma \ref{solvable} shows that $AG^\circ$ can be conjugated inside $\Upp_d(\C)$ by an element of $\GL_d(\C)$. Let $H$ be the  Zariski-closure of $\Gamma \cap AG^\circ$.

Since $H \cap \Gamma$ has finite index in $\Gamma$, it is not virtually nilpotent.  Now by Lemma \ref{mapaffine} there is a homomorphism  $\rho: H  \to \Aff(\C)$ with non virtually nilpotent image. In fact $H$ has index at most $d!J(d)$ in $\Gamma$, so setting $S_0=S^{d!J(d)}$, we see that $S_0H=\Gamma$, and we may apply Lemma \ref{finiteindex} to conclude that $\langle S_0^3 \cap H \rangle$ has non virtually nilpotent image under $\rho$. Hence Corollary \ref{affcor} shows that $\rho_{S_0^3 \cap H} \geq \frac{1}{9}M_\lambda$ for some $\lambda \neq 0$ not a root of unity. The result follows immediately since $\rho_S \geq \frac{1}{3d!J(d)} \rho_{S_0^3}$.
\end{proof}

\begin{remark} We note that if $K_S$ denotes the field  generated by the matrix entries of each element $s \in S$, then $\lambda$ in Theorem \ref{solgrowth} can be found in a finite extension of degree at most $d!$ over $K_S$.  To see this we only need to keep track of the field of definition at every step in the previous argument. In brief, by Remark \ref{rati} the group $F$ can be chosen in $G(K_S)$, then the  Zariski-closure $H$ of $\Gamma \cap AG^\circ$ will be defined over $K_S$ and can be triangularized by some element of $\GL_d(\C)$. It follows that there is a field extension $\widehat{K}_S$ of $K_S$ with degree at most $d!$ such that $H$ can be triangularized by an element of $\GL_d(\widehat{K}_S)$. From Remark \ref{ration} the homomorphism  $\rho: H  \to \Aff(\C)$ given by Lemma \ref{mapaffine} will then be defined over $\widehat{K}_S$. The rest of the proof invoking Lemma \ref{finiteindex} and  Corollary \ref{affcor} is identical, except we make use of Remark \ref{cor-rati} to guarantee that $\lambda$ belongs to $\widehat{K}_S^\times$. 
\end{remark}

\begin{proof}[Proof of Corollary \ref{corq}] From Theorem \ref{solgrowth} and the previous remark, we have $\rho_S  \geq\frac{1}{27d!J(d)} \log  M_\lambda.$ for some $\lambda$ with degree at most $d!$ over $\Q$. The conclusion follows then easily from Dobrowolski's lower bound on the Mahler measure of algebraic numbers \cite{dobrowolski} (in fact a much weaker bound is enough) and the cited bounds for $J(d)$ (see Remark \ref{jordan-bound}). 
\end{proof}
  
\section{Further directions and open problems}

We already mentioned in the introduction that given $\eps>0$ there exist a finitely generated group $\Gamma=\langle S \rangle$ such that $0< \rho_S < \eps$.

These examples of groups with slow exponential growth were constructed by Grigorchuk and de la Harpe in \cite{grigorchuk-delaharpe} out of a presentation for the Grigorchuk group of intermediate growth. They are virtually a product of finitely many free groups. Taking a suitable quotient such groups can be made solvable, in fact even metabelian-by-(finite $2$-group), as was shown in \cite{bartholdi-cornulier}. However the solvability class in these examples is not bounded.

We record here the following:

\begin{problem}
Given $r \in \N$, is there $c_r>0$ such that if $S$ is a finite generating subset of a solvable group with solvability class bounded by $r$, then either $\rho_S=0$ and $\langle S \rangle$ is virtually nilpotent, or $\rho_S > c_r$?
\end{problem}

In view of Lemma \ref{lem:counting} (or \cite{breuillard-solvable}*{\S 7}) a positive answer to this question implies the Lehmer conjecture. It would be nice to investigate, as we did in this paper for linear groups, whether the converse holds as well. To that end we state the following version of the previous problem.

\begin{problem}
Given $r \in \N$, is there $c_r>0$ such that if $S$ is a finite generating subset of a solvable group with solvability class bounded by $r$, then either $\rho_S=0$ and $\langle S \rangle$ is virtually nilpotent, or there is a number $\lambda \in \overline\Q^\times$ not a root of unity such that
$$\rho_S > c_r \log M_\lambda?$$
\end{problem}

The problem obviously reduces to the case when the group is just not virtually nilpotent in the sense that every proper quotient of the group is virtually nilpotent. Such groups are known to be virtually metabelian \cites{groves, breuillard-solvable}. Those that are metabelian (i.e. $r=2$) embed in $\Aff(K)$ for some field $K$, hence for those the answer to the above problem is positive and the proof is easy. However to handle to case when $r>2$, one needs new ideas to overcome the finite index issue.

Another interesting question is whether there are some numbers $\l$ such that $h_\l=\log M_\l$.
In particular, the case of Salem numbers would be very interesting, because in that case
$h_\l=\log M_\l=\log\l^{-1}$ is equivalent to $\dim\mu_\l=1$ by \eqref{hoch}.

\begin{problem}
Is it true that $h_\l=\log \l^{-1}$  for all Salem numbers $\l\in(1/2,1)$?
\end{problem}

Observe that Salem numbers have conjugates on the unit circle, hence Proposition \ref{prp:strict} does not apply.
However, we learnt from Paul Mercat \cite{Mercat-private} that
there are examples outside the scope of Proposition  \ref{prp:strict}, such that $h_\l<\log M_\l$.
Such examples are the roots of the polynomial $x^6 + x^5 + x^4 - x^3 + x^2 + x + 1$, which are
three pairs of complex conjugates that are inside, on and outside the unit circle, respectively, and $M_\l<2$.
Mercat showed that $\rho_\l<\log M_\l$ by computing the first $8$ steps of the random walk and
finding that $|\Supp(\mu_\l^{(8)})|<M_\l^8$.

\bibliographystyle{abbrv}
\bibliography{bibfile}

\end{document}